\documentclass[12pt]{article}
\RequirePackage[OT1]{fontenc}
\RequirePackage{amsthm,amsmath}
\RequirePackage[authoryear]{natbib}

\usepackage[]{hyperref} 
\usepackage{color, dsfont}
\usepackage{url}
\usepackage{amsmath,amsfonts,amsthm,amssymb,amsbsy,bbm}
\usepackage{multirow}
\usepackage{listings}
\usepackage[utf8]{inputenc}
\usepackage{amsmath}
\usepackage{amssymb}
\usepackage{bbm}
\usepackage{amsthm}
\usepackage{algorithm}
\usepackage{algorithmic}
\usepackage{graphicx}
\usepackage{array}
\usepackage{tablefootnote}
\usepackage{amssymb} 
\usepackage[dvipsnames]{xcolor}
\usepackage[authoryear]{natbib}
\usepackage{array,booktabs,multirow,cleveref}
\usepackage{diagbox}

\usepackage{sectsty, secdot}
\sectionfont{\fontsize{12}{14pt plus.8pt minus .6pt}\selectfont}
\renewcommand{\theequation}{\thesection.\arabic{equation}}
\subsectionfont{\fontsize{12}{14pt plus.8pt minus .6pt}\selectfont}

\newtheorem{thm}{Theorem}[section]

\newtheorem{lemma}{Lemma}
\newtheorem{corollary}{Corollary}

\theoremstyle{definition}

\begin{document}

\numberwithin{equation}{section}

\title{\large\bf Identifying the Most Appropriate Order for Categorical Responses}
\author{Tianmeng Wang and Jie Yang\\
		University of Illinois at Chicago}
	
\maketitle
	
\begin{abstract}
Categorical responses arise naturally within various scientific disciplines. In many circumstances, there is no predetermined order for the response categories, and the response has to be modeled as nominal. In this study, we regard the order of response categories as part of the statistical model, and show that the true order, when it exists, can be selected using likelihood-based model selection criteria. For predictive purposes, a statistical model with a chosen order may outperform models based on nominal responses, even if a true order does not exist. For multinomial logistic models, widely used for categorical responses, we show the existence of theoretically equivalent orders that cannot be differentiated based on likelihood criteria, and determine the connections between their maximum likelihood estimators. We use simulation studies and a real-data analysis to confirm the need and benefits of choosing the most appropriate order for categorical responses.
\end{abstract}
	

\section{Introduction}\label{sec:introduction}

Categorical responses, in which the measurement scale consists of a set of categories, arise naturally in many scientific disciplines. Examples include the social sciences for measuring attitudes and opinions, health sciences for measuring responses to a medical treatment, behavioral sciences for diagnosing mental illness, ecology for determining primary land use in satellite images, education for measuring student responses, marketing for  determining consumer preferences, among many others \citep{agresti2018introduction}. When the response is binary, generalized linear models are widely used \citep{pmcc1989, dobson2018}. When responses have three or more categories, multinomial logistic models are popular \citep{pmcc1995, atkinson1999, bu2020}, and include four kinds of logit models:  {\it baseline-category}, {\it cumulative}, {\it adjacent-categories}, and {\it continuation-ratio} logit models. 

Following the notation of \cite{bu2020}, we consider summarized data in the form of $\{({\mathbf x}_i, {\mathbf Y}_i), i=1, \ldots, m\}$ from an experiment or observational study with $d\geq 1$ covariates and $J\geq 3$ response categories, where ${\mathbf x}_i = (x_{i1}, \ldots, x_{id})^T$, for $i=1, \ldots, m$, are distinct level combinations of the $d$ covariates, and ${\mathbf Y}_i=(Y_{i1},\cdots,Y_{iJ})^T$, with $Y_{ij}$ indicating the number of original observations associated with the covariates ${\mathbf x}_i$ and the $j$th response category, for $j=1, \ldots, J$. A multinomial logistic model assumes ${\mathbf Y}_i \sim {\rm Multinomial}(n_i; \pi_{i1},\cdots,\pi_{iJ})$ independently, with $n_i = \sum_{j=1}^J Y_{ij} > 0$ and positive categorical probabilities $\pi_{ij}$ associated with ${\mathbf x}_i$, for $i=1, \ldots, m$.

For {\it nominal} responses, that is, the response categories do not have a natural ordering \citep{agresti2013categorical}, baseline-category logit models, also known as multiclass logistic regression models, are commonly used. Following \cite{bu2020}, the baseline-category logit model with partial proportional odds ({\it ppo}) can be described in general as
\begin{equation}\label{eq:model_baseline}
    	\log\left(\frac{\pi_{ij}}{\pi_{iJ}}\right) = \eta_{ij} = {\mathbf h}_j^T({\mathbf x}_i)\boldsymbol\beta_j+{\mathbf h}_c^T({\mathbf x}_i)\boldsymbol\zeta \ ,
\end{equation}
where ${\mathbf h}_j^T(\cdot) = (h_{j1}(\cdot), \ldots, h_{jp_j}(\cdot))$ and ${\mathbf h}_c^T(\cdot) = (h_{1}(\cdot), \ldots, h_{p_c}(\cdot))$ are known predictor functions, and $\boldsymbol\beta_j = (\beta_{j1}, \ldots, \beta_{jp_j})^T$ and $\boldsymbol\zeta = (\zeta_{1}, \ldots, \zeta_{p_c})^T$ are unknown regression parameters, for $i=1, \ldots, m$, $j=1, \ldots, J-1$. As special cases, ${\mathbf h}_j^T({\mathbf x}_i) \equiv 1$ leads to a proportional odds ({\it po}) model that assumes the same parameters for all categories, except the intercepts, and ${\mathbf h}_c^T({\mathbf x}_i) \equiv 0$ leads to a nonproportional odds ({\it npo}) model that allows all parameters to change across categories. For additional explanations and examples about ppo, po, and npo models, see \cite{bu2020}.  

In model~\eqref{eq:model_baseline}, the $J$th category is treated as the baseline category. It is well known that the choice of baseline category does not matter, because the resulting models are equivalent (e.g., see Section~4.4 in \cite{hastie2009elements}). However, the equivalence of choices of baseline categories is true only for npo models. As we show in Section~\ref{sec:baseline_category}, for po or general ppo models, those with different baseline categories are not equivalent, and thus the baseline category should be chosen carefully.

The other three logit models assume that the response categories have a natural ordering or a hierarchical structure, and are known as {\it ordinal} or {\it hierarchical} models, respectively. Following \cite{bu2020}, these three logit models with ppo can be written as follows:
\begin{equation}\label{eq:cumulative_model}
  \log\left(\frac{\pi_{i1}+\cdots + \pi_{ij}}{\pi_{i,j+1}+\cdots + \pi_{iJ}}\right) = \eta_{ij} = {\mathbf h}_j^T({\mathbf x}_i)\boldsymbol\beta_j+{\mathbf h}_c^T({\mathbf x}_i)\boldsymbol\zeta~, \>\mbox{cumulative} 
\end{equation}
\begin{equation}\label{eq:adjacent_model}
  \log\left(\frac{\pi_{ij}}{\pi_{i,j+1}}\right) = \eta_{ij} = {\mathbf h}_j^T({\mathbf x}_i)\boldsymbol\beta_j+{\mathbf h}_c^T({\mathbf x}_i)\boldsymbol\zeta~, \>\mbox{adjacent categories}
\end{equation}
\begin{equation}\label{eq:continuation_model}
	\log\left(\frac{\pi_{ij}}{\pi_{i,j+1} + \cdots + \pi_{iJ}}\right) = \eta_{ij} = {\mathbf h}_j^T({\mathbf x}_i)\boldsymbol\beta_j+{\mathbf h}_c^T({\mathbf x}_i)\boldsymbol\zeta~, \>\mbox{continuation ratio.} 
\end{equation}
These are special cases of the multivariate generalized linear models \citep{mccullagh1980regression} or multivariate logistic models \citep{pmcc1995}.

Note that cumulative logit models have been extended to cumulative link models and ordinal regression models \citep{mccullagh1980regression, agresti2013categorical, ytm2016}. 
A baseline-category logit model can be modified with a probit link, and is known as a multinomial probit model \citep{aitchison1970, agresti2013categorical, greene2018econometric}. Furthermore, the continuation-ratio logit model can be changed with a complementary log-log link \citep{oconnell2006} for data analysis. We focus on multinomial logistic models, because the logit link is the most commonly used.

For some applications, the ordering of the response categories is clear. For example, trauma data \citep{chuang1997, agresti2010analysis, bu2020} includes $J=5$ ordinal response categories, namely, {\tt death}, {\tt vegetative} {\tt state}, {\tt major} {\tt disability}, {\tt minor} {\tt disability}, and {\tt good} {\tt recovery}, known as the Glasgow Outcome Scale \citep{jennett1975}. A cumulative logit model with npo has been recommended for modeling such data \citep{bu2020}.

For some other applications, the ordering is either unknown or difficult to determine. As a motivating example, the police data described in Section~\ref{sec:real_data} contain covariates about individuals killed by the police in the United States for the period 2000 to 2016. 
The responses have four categories, namely, {\tt shot}, {\tt tasered}, {\tt shot and tasered}, and {\tt other}. To model the responses that are relevant to the police's actions on various covariates of the suspects, one strategy is to treat the response as nominal and use the baseline-category logit model~\eqref{eq:model_baseline}. Another strategy is to determine an appropriate order for the categories, and then to use one of the other three logit models~\eqref{eq:cumulative_model}, \eqref{eq:adjacent_model}, and \eqref{eq:continuation_model}. Our analysis in Section~\ref{sec:real_data} shows that a continuation-ratio npo model with a chosen order performs best, and that the second strategy is significantly better.

A critical question that needs answering is whether we can identify the true order of the response categories, when it exists. Our answer is summarized as follows. First, we will show in Section~\ref{sec:optimal_true_order} that if there is a true order with a true model, it will attain the maximum likelihood asymptotically, so that it can be identified using a likelihood-based model selection technique, such as the AIC or BIC. Second, depending on the type of logit model, some orders are indistinguishable or equivalent, because they attain the same maximum likelihood (see Table~\ref{tab:summary} for a summary of the equivalence among the orders identified in Section~\ref{sec:order_equivalence}). Third, depending on the range of covariates or predictors, some order that is not equivalent to the true one may approximate the maximum likelihood so well that it is not numerically distinguishable from the true order (see Section~\ref{sec:order_x_i}).

\begin{table}[ht]
\caption{\label{tab:summary}Equivalence among Orders of Response Categories}
\begin{center}
\begin{tabular}{ccc}\hline
     {\bf Logit model} & {\bf ppo or po} & {\bf npo}\\ \hline
     Baseline-category & Same if the baseline is unchanged  & All orders are the same \\
     & (Theorem~\ref{thm:baseline_ppo}) & (Theorem~\ref{thm:baseline_npo})\\ 
     Cumulative & Same as its reversed order  & Same as its reversed order \\
     & (Theorem~\ref{thm:cumulative_order}) & (Theorem~\ref{thm:cumulative_order}) \\ 
     Adjacent-categories & Same as its reversed order  & All orders are the same \\
     & (Theorem~\ref{thm:adjacent_ppo}) & (Theorem~\ref{thm:adjacent_npo}) \\  
     Continuation-ratio & All orders are distinguishable  & Same if switching last two \\
     & (Section~\ref{sec:simulation_trauma}) & (Theorem~\ref{thm:continuation_npo}) \\  \hline
    \end{tabular}
\end{center}
\end{table}

In practice, nevertheless, even there is no true order among the response categories, we can still use likelihood-based model selection techniques to choose a working order supported by the data. As such, an ordinal model based on the working order will outperform nominal models in terms of prediction accuracy (see Sections~\ref{sec:baseline_category} and \ref{sec:order_nominal}). We provide a real-data example in Section~\ref{sec:real_data} that shows how to reduce the prediction errors significantly based on the working order. Overall, we suggest that practitioners view identifying the most appropriate order of response categories as part of the model selection procedure.

\section{Equivalence among Orders of Response Categories}\label{sec:order_equivalence}

In this section, before investigating which order of the response categories is best, we first answer a more fundamental question that occurs when two different orders lead to the same maximum likelihood. In this case, if one uses the AIC or BIC to select the best model (e.g., see \cite{hastie2009elements} for a good review), these two orders are indistinguishable, or equivalent. Such a phenomenon has been observed for some baseline-category models~\citep{hastie2009elements}, and here we show that it exists fairly generally in other multinomial logistic models as well (see Table~\ref{tab:summary} for a summary).

Given the original data $\{({\mathbf x}_i, {\mathbf Y}_i), i=1, \ldots, m\}$, ${\mathbf Y}_i = (Y_{i1}, \ldots, Y_{iJ})^T$ consists of the counts of observations falling into the response categories in the original order or labels $\{1, \ldots, J\}$. If we consider a regression model with a different order $\{\sigma(1), \ldots, \sigma(J)\}$ of the response categories, where $\sigma: \{1, \ldots, J\} \rightarrow \{1, \ldots, J\}$ is a permutation, this is equivalent to fitting the model using the permuted data $\{({\mathbf x}_i, {\mathbf Y}_i^\sigma), i=1, \ldots, m\}$,  where ${\mathbf Y}_i^\sigma = (Y_{i\sigma(1)}, \ldots,$ $ Y_{i\sigma(J)})^T$. We denote ${\mathcal P}$ as the collection of all permutations on $\{1, \ldots, J\}$. Then, each permutation $\sigma \in {\mathcal P}$ represents an order of the response categories, also denoted by $\sigma$.

Now, we consider two orders or permutations $\sigma_1, \sigma_2 \in {\mathcal P}$. For $i=1,2$, we denote $l_i(\boldsymbol\theta)$ as the likelihood function with order $\sigma_i$ or $\{\sigma_i(1), \ldots, \sigma_i(J)\}$. We say that $\sigma_1$ and $\sigma_2$ are {\it equivalent}, denoted as $\sigma_1 \sim \sigma_2$, if $\max_{\boldsymbol\theta\in \boldsymbol\Theta} l_1(\boldsymbol\theta) = \max_{\boldsymbol\theta\in \boldsymbol\Theta} l_2 (\boldsymbol\theta)$, where $\boldsymbol\Theta$ is the parameter space, that is, the set of all feasible $\boldsymbol\theta$. It is straightforward that ``$\sim$'' is an {\it equivalence relation} among the permutations in ${\mathcal P}$. That is, $\sigma \sim \sigma$ for all $\sigma$; $\sigma_1\sim \sigma_2$ if and only if $\sigma_2 \sim \sigma_1$; and $\sigma_1 \sim \sigma_2$ and $\sigma_2 \sim \sigma_3$ imply $\sigma_1 \sim \sigma_3$ (e.g., see Section~1.4 in \cite{wallace1998groups}).

\subsection{Partial proportional odds (ppo) models}\label{sec:ppo_order}

According to \cite{bu2020} (see also \cite{pmcc1995, atkinson1999}), all four logit models~\eqref{eq:model_baseline}, \eqref{eq:cumulative_model}, \eqref{eq:adjacent_model}, and \eqref{eq:continuation_model} with ppo (that is, the most general case) can be rewritten in a unified form
\begin{equation}\label{eq:logitunifiedmodel}
{\mathbf C}^T\log({\mathbf L}{{\boldsymbol\pi}}_i)={\boldsymbol\eta}_i={\mathbf X}_i{\boldsymbol\theta}, \qquad     i=1,\cdots,m,
\end{equation}
where ${\mathbf C}$ is a $J\times (2J-1)$ constant matrix, ${\mathbf L}$ is a $(2J-1)\times J$ constant matrix, depending on the type of logit model, ${\boldsymbol\pi}_i = (\pi_{i1}, \ldots, \pi_{iJ})^T$ are category probabilities at ${\mathbf x}_i$ satisfying $\sum_{j=1}^J \pi_{ij}=1$, $\boldsymbol\eta_i = (\eta_{i1}, \ldots, \eta_{iJ})^T$ are the linear predictors, ${\mathbf X}_i$ is the model matrix consisting of ${\mathbf h}_j({\mathbf x}_i)$ and ${\mathbf h}_c({\mathbf x}_i)$, and $\boldsymbol\theta = (\boldsymbol\beta_1^T, \ldots, \boldsymbol\beta_{J-1}^T, \boldsymbol\zeta^T)^T$ consists of $p=p_1 + \cdots + p_{J-1} + p_c$ regression parameters. Please see \cite{bu2020} for more details and examples.

According to Theorem~5.1 in \cite{bu2020}, for cumulative logit models, the parameter space $\boldsymbol\Theta = \{\boldsymbol\theta \in R^p \mid {\mathbf h}_j^T({\mathbf x}_i) \boldsymbol\beta_j < {\mathbf h}_{j+1}^T ({\mathbf x}_i) \boldsymbol\beta_{j+1}, \mbox{ for }j=1, \ldots, J-2, i=1, \ldots, m\}$ depends on the range of covariates. For the other three logit models, $\boldsymbol\Theta$ is typically $R^p$ itself. Apparently, neither the parameter space $\boldsymbol\Theta$ nor the model~\eqref{eq:logitunifiedmodel} is affected by a permutation of the data ${\mathbf Y}_i$~.

By reorganizing the formulae in Section~S.11 of the Supplementary Material of \cite{bu2020}, we write the category probabilities $\pi_{ij}$ as explicit functions of $\eta_{ij}$ (and thus of $\boldsymbol\theta$) in Lemma~\ref{lem:pi_as_eta}. To simplify the notation, we denote $\rho_{ij} = {\rm logit}^{-1}(\eta_{ij}) = e^{\eta_{ij}}/(1+e^{\eta_{ij}})$, and thus $\rho_{ij}/(1-\rho_{ij}) = e^{\eta_{ij}}$, for $j=1, \ldots, J-1$, and $\rho_{i0}\equiv 0$, for $i=1, \ldots, m$.

\begin{lemma}\label{lem:pi_as_eta}
For the four logit models \eqref{eq:model_baseline}, \eqref{eq:cumulative_model}, \eqref{eq:adjacent_model}, and \eqref{eq:continuation_model},
\begin{equation}\label{eq:pi_ij}
\pi_{ij} = \left\{
\begin{array}{cl}
\frac{\frac{\rho_{ij}}{1-\rho_{ij}}}{1+\sum_{l=1}^{J-1} \frac{\rho_{il}}{1-\rho_{il}}} & \mbox{, baseline category}\\
\rho_{ij}-\rho_{i,j-1} & \mbox{, cumulative}\\
\frac{\prod_{l=j}^{J-1} \frac{\rho_{il}}{1-\rho_{il}}}{1 + \sum_{s=1}^{J-1} \prod_{l=s}^{J-1} \frac{\rho_{il}}{1-\rho_{il}}} & \mbox{, adjacent categories}\\
\prod_{l=0}^{j-1} (1-\rho_{il}) \cdot \rho_{ij} & \mbox{, continuation ratio,}
\end{array}
\right.
\end{equation}
for $i=1, \ldots, m$ and $j=1, \ldots, J-1$. In addition, 
\begin{equation}\label{eq:pi_iJ}
\pi_{iJ} = \left\{
\begin{array}{cl}
\frac{1}{1+\sum_{l=1}^{J-1} \frac{\rho_{il}}{1-\rho_{il}}} & \mbox{, baseline category}\\
1-\rho_{i,J-1} & \mbox{, cumulative}\\
\frac{1}{1 + \sum_{s=1}^{J-1} \prod_{l=s}^{J-1} \frac{\rho_{il}}{1-\rho_{il}}} & \mbox{, adjacent categories}\\
\prod_{l=1}^{J-1} (1-\rho_{il}) & \mbox{, continuation ratio,}
\end{array}
\right.
\end{equation}
for $i=1, \ldots, m$.
\end{lemma}

Because $\eta_{ij} = {\mathbf h}_j^T({\mathbf x}_i)\boldsymbol\beta_j+{\mathbf h}_c^T({\mathbf x}_i)\boldsymbol\zeta$, Lemma~\ref{lem:pi_as_eta} indicates that $\pi_{ij}$ are functions of ${\mathbf x}_i$ and $\boldsymbol\theta$, and do not depend on ${\mathbf Y}_i$ or ${\mathbf Y}_i^\sigma$, which is true for general multinomial logit models~\eqref{eq:logitunifiedmodel}. The following theorem provides a sufficient condition for $\sigma_1 \sim \sigma_2$.

\begin{thm}\label{thm:permutation} Consider the multinomial logit model~\eqref{eq:logitunifiedmodel} with independent observations and two permutations, $\sigma_1, \sigma_2\in {\mathcal P}$. Suppose for any $\boldsymbol\theta_1 \in \boldsymbol\Theta$, there exists a $\boldsymbol\theta_2 \in \boldsymbol\Theta$, and vice versa, such that,
\begin{equation}\label{eq:permutation_pi}
	\pi_{i\sigma_1^{-1}(j)}(\boldsymbol\theta_1) = \pi_{i\sigma_2^{-1}(j)}(\boldsymbol\theta_2) ,
\end{equation}
for all $i=1, \ldots, m$ and $j=1, \ldots, J$. Then, $\sigma_1 \sim \sigma_2$. Furthermore, $\sigma \sigma_1 \sim \sigma \sigma_2$~, for any $\sigma \in {\mathcal P}$.
\end{thm}

Here, $\sigma\sigma_1$ in Theorem~\ref{thm:permutation} represents the composition of $\sigma$ and $\sigma_1$. That is, $\sigma\sigma_1(j) = \sigma(\sigma_1(j))$, for all $j$. Using this notation, $(\sigma\sigma_1)^{-1} = \sigma_1^{-1}\sigma^{-1}$.

The proof of Theorem~\ref{thm:permutation} is relegated to the Supplementary Material.

\begin{thm}\label{thm:baseline_ppo}
    Consider the baseline-category logit model~\eqref{eq:model_baseline} with ppo. Suppose ${\mathbf h}_1({\mathbf x}_i) = \cdots = {\mathbf h}_{J-1}({\mathbf x}_i)$, for all $i=1, \ldots, m$. Then, all orders of response categories that keep $J$ invariant are equivalent.
\end{thm}

Theorem~\ref{thm:baseline_ppo} includes baseline-category logit models with po, because ${\mathbf h}_j({\mathbf x}_i) \equiv 1$ for po models. It actually includes many npo or ppo models used in practice, where ${\mathbf h}_1 = \cdots = {\mathbf h}_{J-1}$. For example, main-effects models that assume ${\mathbf h}_1({\mathbf x}_i) = \cdots = {\mathbf h}_{J-1}({\mathbf x}_i) = (1,{\mathbf x}_i^T)^T$ are widely used. 

Theorem~\ref{thm:baseline_ppo} also implies that if $\sigma(J) \neq J$, then $\sigma$ may not be equivalent to the original order ${\rm id}$, or the identity permutation. We provide such a numerical example in Section~\ref{sec:baseline_category}.

\begin{thm}\label{thm:cumulative_order}
    Consider the cumulative logit model~\eqref{eq:cumulative_model} with ppo. Suppose ${\mathbf h}_j({\mathbf x}_i) = {\mathbf h}_{J-j}({\mathbf x}_i)$, for all $i=1, \ldots, m$ and $j=1, \ldots, J-1$. Then, any order $\sigma_1$ is equivalent to its reverse order $\sigma_2$, which satisfies $\sigma_2(j) = \sigma_1(J+1-j)$, for $j=1, \ldots, J$. 
\end{thm}

Theorem~\ref{thm:cumulative_order} includes cumulative logit models with po, because ${\mathbf h}_j({\mathbf x}_i) \equiv 1$. It also includes many npo or ppo models used in practice that satisfy ${\mathbf h}_1 = \cdots = {\mathbf h}_{J-1}$~.

\begin{thm}\label{thm:adjacent_ppo}
    Consider the adjacent-categories logit model~\eqref{eq:adjacent_model} with ppo. Suppose ${\mathbf h}_1({\mathbf x}_i) = \cdots ={\mathbf h}_{J-1} ({\mathbf x}_i)$, for all $i=1, \ldots, m$. Then, any order $\sigma_1$ is equivalent to its reverse order $\sigma_2$~, which satisfies $\sigma_2(j) = \sigma_1(J+1-j)$, for all $j=1, \ldots, J$.
\end{thm}

\subsection{Nonproportional odds (npo) models}\label{sec:npo_order_equivalence}

By removing the item ${\mathbf h}_c^T({\mathbf x}_i)\boldsymbol\zeta$ from \eqref{eq:model_baseline}, \eqref{eq:cumulative_model}, \eqref{eq:adjacent_model}, and \eqref{eq:continuation_model}, we obtain  explicit forms of the four logit models with npo. For npo models, $\boldsymbol\theta = (\boldsymbol\beta_1^T, \ldots,$ $\boldsymbol\beta_{J-1}^T)^T$, $p=p_1 + \cdots + p_{J-1}$~, and $\eta_{ij} = {\mathbf h}_j^T({\mathbf x}_i)\boldsymbol\beta_j$, for $i=1, \ldots, m$ and $j=1, \ldots, J-1$. Compared with po models, npo models involve more regression parameters, and thus are more flexible. For more details about multinomial logistic models with npo, please see, for example, Section~S.8 in the Supplementary Material of \cite{bu2020}.

\begin{thm}\label{thm:baseline_npo}
    Consider the baseline-category logit model with npo. Suppose ${\mathbf h}_1({\mathbf x}_i) = \cdots = {\mathbf h}_{J-1}({\mathbf x}_i)$, for all $i=1, \ldots, m$. Then, all orders of response categories are equivalent.
\end{thm}

Theorem~\ref{thm:baseline_npo} confirms that the choice of baseline category does not matter for multiclass logistic regression models \citep{hastie2009elements}. What is new here is the explicit correspondence between $\boldsymbol\theta_1$ and $\boldsymbol\theta_2$ provided in the proof of Theorem~\ref{thm:baseline_npo}. Based on the correspondence, if we obtain the maximum likelihood estimate (MLE) for $\boldsymbol\theta_1$, we can easily derive the MLE for $\boldsymbol\theta_2$ explicitly, without running another optimization.

\begin{thm}\label{thm:adjacent_npo}
    Consider the adjacent-categories logit model with npo. Suppose ${\mathbf h}_1({\mathbf x}_i) = \cdots = {\mathbf h}_{J-1} ({\mathbf x}_i)$, for all $i=1, \ldots, m$. Then, all orders of response categories are equivalent.
\end{thm}

The result of Theorem~\ref{thm:adjacent_npo} is truly surprising. The order of the response categories in an ordinal model does not matter! The transformation~(S.1) and its inverse~(S.3) in the proof of Theorem~\ref{thm:adjacent_npo} in the Supplementary Material are not trivial either. 

\begin{thm}\label{thm:continuation_npo}
    For the continuation-ratio logit model with npo, $\sigma_1 \sim \sigma_2$ if $\sigma_2 = \sigma_1 (J-1,J)$, where $(J-1,J)$ is a transposition that switches $J-1$ and $J$. 
\end{thm}

\section{Asymptotic Optimality of True Order}\label{sec:optimal_true_order}

In this section, we discuss the optimality of the true order of the response categories, when it exists. In short, the model with the true order is asymptotically optimal in terms of an AIC or BIC likelihood-based model selection criterion.

Suppose an experiment is performed under the multinomial logit model \eqref{eq:logitunifiedmodel}, with predetermined design points ${\mathbf x}_1, \ldots, {\mathbf x}_m$~, the true parameter values $\boldsymbol\theta_0 \in \boldsymbol\Theta \subseteq R^p$, and the true order $\sigma_0 \in {\mathcal P}$ of the response categories. Recall that the original experiment assigns $n_i$ subjects to ${\mathbf x}_i$~, with the total number of subjects $n = \sum_{i=1}^m n_i$. To avoid trivial cases, we assume $n_i > 0$ for each $i$ (otherwise, we may delete any ${\mathbf x}_i$ with $n_i=0$). 

In order to consider the asymptotic properties of the parameter and order estimators, we consider independent and identically distributed (i.i.d.) observations $(X_l, Y_l)$, for $l=1, \ldots, N$, generated as follows: (i) $X_1, \ldots, X_N$ are i.i.d.~from a discrete distribution taking values in $\{{\mathbf x}_1, \ldots, {\mathbf x}_m\}$ with probabilities $n_1/n, \ldots, n_m/n$, respectively; (ii) given $X_l = {\mathbf x}_i$, $Y_l$ follows Multinomial$(1; \pi_{i\sigma_0^{-1}(1)}(\boldsymbol\theta_0), \ldots,$ $\pi_{i\sigma_0^{-1}(J)}(\boldsymbol\theta_0))$, that is, $Y_l$ takes values in $\{1,$ $\ldots,$ $J\}$ with probabilities $\pi_{i\sigma_0^{-1}(1)}(\boldsymbol\theta_0), \ldots,$  $\pi_{i\sigma_0^{-1}(J)}(\boldsymbol\theta_0)$, respectively.  The summarized data can still be denoted as $\{({\mathbf x}_i, {\mathbf Y}_i), i=1, \ldots, m\}$, where
\[
{\mathbf Y}_i = (Y_{i1}, \ldots, Y_{iJ})^T \sim {\rm Multinomial}\left(N_i; \pi_{i\sigma_0^{-1}(1)}(\boldsymbol\theta_0), \ldots, \pi_{i\sigma_0^{-1}(J)}(\boldsymbol\theta_0)\right),
\]
and $N_i = \sum_{j=1}^J Y_{ij}$ is the total number of subjects assigned to ${\mathbf x}_i$~.

Given the count data $Y_{ij}$, the log-likelihood function under the multinomial logit model~\eqref{eq:logitunifiedmodel} with parameters $\boldsymbol\theta$ and a permutation $\sigma\in {\mathcal P}$ applied to $Y_{ij}$ to determine the true order is
\begin{eqnarray*}
l(\boldsymbol\theta, \sigma) &=&  \sum_{i=1}^m \sum_{j=1}^J Y_{i\sigma(j)} \log\pi_{ij}(\boldsymbol\theta) + \sum_{i=1}^m \log (N_i!) - \sum_{i=1}^m \sum_{j=1}^J \log (Y_{i\sigma(j)}!)\\
&=&  \sum_{i=1}^m \sum_{j=1}^J Y_{ij} \log\pi_{i\sigma^{-1}(j)}(\boldsymbol\theta) + \sum_{i=1}^m \log (N_i!) - \sum_{i=1}^m \sum_{j=1}^J \log (Y_{ij}!)\\
&=& \sum_{i=1}^m \sum_{j=1}^J Y_{ij} \log\pi_{i\sigma^{-1}(j)}(\boldsymbol\theta) + {\rm constant},
\end{eqnarray*}
Then, the MLE $(\hat{\boldsymbol\theta}, \hat\sigma)$ that maximizes $l(\boldsymbol\theta, \sigma)$ maximizes $$l_N(\boldsymbol\theta, \sigma) = \sum_{i=1}^m \sum_{j=1}^J Y_{ij} \log\pi_{i\sigma^{-1}(j)}(\boldsymbol\theta)$$ as well.  

\begin{lemma}\label{lem:mle_pi_ij}
If $\hat{\boldsymbol\theta} \in \boldsymbol\Theta$ and $\hat{\sigma} \in {\mathcal P}$ satisfy $\pi_{i\hat\sigma^{-1}(j)}(\hat{\boldsymbol\theta}) = Y_{ij}/N_i$ for all $i$ and $j$, then $(\hat{\boldsymbol\theta}, \hat\sigma)$ must be an MLE.
\end{lemma}

To explore the asymptotic properties of log-likelihood and MLEs, we denote $l_0 = \sum_{i=1}^m \sum_{j=1}^J n_i \pi_{ij}(\boldsymbol\theta_0) \log\pi_{ij}(\boldsymbol\theta_0)/n  \in (-\infty, 0)$, which is a finite constant. We further denote $n_0 = \min\{n_1, \ldots, n_m\} \geq 1$ and $\pi_0 = \min\{\pi_{i\sigma_0^{-1}(j)}(\boldsymbol\theta_0), i=1, \ldots, m; j=1, \ldots, J\} \in (0, 1)$.

\begin{lemma}\label{lem:limit_l_N}
As $N\rightarrow \infty$, $N^{-1} l_N(\boldsymbol\theta_0, \sigma_0) \rightarrow l_0 < 0$ almost surely.
\end{lemma}

Let $(\hat{\boldsymbol\theta}_N, \hat{\sigma}_N)$ denote an MLE that maximizes $l_N(\boldsymbol\theta, \sigma)$ or, equivalently, $l(\boldsymbol\theta, \sigma)$. The following theorem indicates that the true values $(\boldsymbol\theta_0, \sigma_0)$ asymptotically maximize $l_N(\boldsymbol\theta, \sigma)$ as well. That is, the true parameter values $\boldsymbol\theta_0$ and the true order $\sigma_0$ are asymptotically optimal in terms of the likelihood principle.

\begin{thm}\label{thm:consist_order}
As $N\rightarrow \infty$, $N^{-1}|l_N(\hat{\boldsymbol\theta}_N, \hat{\sigma}_N) - l_N(\boldsymbol\theta_0, \sigma_0)| \rightarrow 0$ almost surely. Furthermore, $N^{-1} l_N(\hat{\boldsymbol\theta}_N, \hat{\sigma}_N) \rightarrow l_0$ and $l_N(\boldsymbol\theta_0, \sigma_0)/l_N(\hat{\boldsymbol\theta}_N, \hat{\sigma}_N) \rightarrow 1$ almost surely as well.
\end{thm}

Following \cite{burnham2004aic}, we define 
\begin{eqnarray*}
    {\rm AIC}(\boldsymbol\theta, \sigma) &=& -2 l(\boldsymbol\theta, \sigma) + 2p\\
    {\rm BIC}(\boldsymbol\theta, \sigma) &=& -2 l(\boldsymbol\theta, \sigma) + (\log N) p .
\end{eqnarray*}
Then, the usual ${\rm AIC} = {\rm AIC}(\hat{\boldsymbol\theta}_N, \hat{\sigma}_N) \leq {\rm AIC}(\boldsymbol\theta_0, \sigma_0)$, and the usual ${\rm BIC} = {\rm BIC}(\hat{\boldsymbol\theta}_N, \hat{\sigma}_N) \leq {\rm BIC}(\boldsymbol\theta_0, \sigma_0)$. As a direct conclusion of Theorem~\ref{thm:consist_order}, we have the following corollary.

\begin{corollary}\label{cor:AIC_BIC}
$N^{-1}|{\rm AIC} - {\rm AIC}(\boldsymbol\theta_0, \sigma_0)| = N^{-1}|{\rm BIC} - {\rm BIC}(\boldsymbol\theta_0, \sigma_0)| \rightarrow 0$ almost surely, as $N\rightarrow \infty$.
\end{corollary}

Theorem~\ref{thm:consist_order} and Corollary~\ref{cor:AIC_BIC} confirm that the true parameter and order $(\boldsymbol\theta_0, \sigma_0)$ are among the best options asymptotically under likelihood-based model selection criteria. Nevertheless, note that there are two cases where one may not be able to identify the true order $\sigma_0$ easily.

{\it Case one:} If there exists another order $\sigma\sim \sigma_0$, as discussed in Section~\ref{sec:order_equivalence}, then there exists another $\boldsymbol\theta$ such that $\pi_{i\sigma^{-1}(j)}(\boldsymbol\theta) = \pi_{i\sigma_0^{-1}(j)}({\boldsymbol\theta_0})$, for all $i$ and $j$. In this case, $l(\boldsymbol\theta, \sigma) = l(\boldsymbol\theta_0, \sigma_0)$, and the true order $\sigma_0$ with $\boldsymbol\theta_0$ is not distinguishable from the order $\sigma$ with ${\boldsymbol\theta}$.

{\it Case two:} In practice, given the set of experimental settings $\{{\mathbf x}_1, \ldots, {\mathbf x}_m\}$, it is not unlikely, especially when the range of experimental settings is narrow, that there exists a $(\boldsymbol\theta, \sigma)$ such that $\pi_{i\sigma^{-1}(j)}(\boldsymbol\theta) \approx \pi_{i\sigma_0^{-1}(j)}({\boldsymbol\theta_0})$, for all $i$ and $j$. Then, the order $\sigma$ with ${\boldsymbol\theta}$ achieves roughly the same likelihood.  
In this case, the difference between $\sigma_0$ and $\sigma$ could be insignificant, and may not improve unless one increases the sample size $N$ and the range of $\{{\mathbf x}_1, \ldots, {\mathbf x}_m\}$ (e.g., see Sections~\ref{sec:simulation_trauma} and \ref{sec:order_x_i}).
Compared with po models, npo models have more parameters and are more flexible. As a result, Case~two may occur more often for npo models.

In both Case~one and Case~two, the true order $\sigma_0$ is not significantly better than all other orders, even with an increased sample size $N$. Nevertheless, the results of our simulation studies (see Section~\ref{sec:simulation_studies}) show that there are two situations in which the true order $\sigma_0$ can be identified easily. One is with larger absolute values of the regression coefficients (e.g., see Section~\ref{sec:order_treatment_effect}). The other is with a larger range of experimental settings (see Section~\ref{sec:order_x_i}). Both situations reduce the possibility of Case~two, but neither can fix Case~one.

\section{Simulation Studies}\label{sec:simulation_studies}

\subsection{Simulation study based on trauma data}\label{sec:simulation_trauma}

In the example mentioned in Section~\ref{sec:introduction}, the trauma data (see Table~1 in \cite{chuang1997} or Table~7.6 in \cite{agresti2010analysis}) are the responses of $N=802$ trauma patients from four treatment groups under different dose levels. The response categories (GOS) have a clear order, namely, {\tt death} (1), {\tt vegetative state} (2), {\tt major disability} (3), {\tt minor disability} (4), and {\tt good recovery} (5). That is, the original order $\sigma_0={\rm id}$ is the true order. In this case, $J=5$. For illustration purposes, the only predictor $x$ is chosen as the dose level, taking values in $\{1,2,3,4\}$.

In this section, we present simulation studies to explore whether we can identify the true order of the response categories under different multinomial logit models. For each model, for example, the baseline-category logit model with po (see {\it baseline-category po} in Table~\ref{tab:trauma:ori_size}), (i) we fit the model using the original data $(Y_{ij})_{ij}$ against the dose level $x_i=i$ to obtain estimated parameter values $\hat{\boldsymbol\theta}_o$; (ii) we simulate a new data set $(Y'_{ij})_{ij}$ with total sample size $N$ using the model with regression coefficients $\hat{\boldsymbol\theta}_o$; that is, for the simulated data, we have a true order $\sigma_0={\rm id}$ and true parameter values $\boldsymbol\theta_0=\hat{\boldsymbol\theta}_o$; (iii) for each possible order $\sigma \in {\mathcal P}$, we fit the model using permuted simulated data $(Y'_{i\sigma(j)})_{ij}$ and calculate the corresponding AIC value; and (iv) we check the difference between the AIC at $\sigma_0$ (denoted as AIC$_0$) and the AIC at $\hat\sigma$ (denoted as AIC$_*$) that minimizes the AIC value, as well as the rank of AIC$_0$ among all orders. Ideally, we have AIC$_0-$AIC$_* = 0$ with rank $1$ out of $5!=120$; 
that is, the original order achieves the smallest AIC value. Nevertheless, in practice, the result depends on the sample size $N$ and the set of experimental settings.

\begin{table}[ht]
\caption{\label{tab:trauma:ori_size}Trauma Simulation Study with $N=802$}
\begin{center}
\begin{tabular}{c | c c | c c}
\hline
& \multicolumn{2}{c|}{True Order} & \multicolumn{2}{c}{Best Order} \\
\cline {2-5}
Model & \textsc{aic}$_0$ & Rank & \textsc{aic}$_*$ & \textsc{aic}$_0$$-$\textsc{aic}$_*$\\
\hline
Baseline-category po & 96.77 & 73 & 96.40 & 0.37\\
Cumulative po & 94.54 & 7 & 93.84 & 0.70\\
Cumulative npo & 102.15 & 51 & 101.70 & 0.45\\
Adjacent-categories po & 92.66 & 3 & 92.08 & 0.58\\
Continuation-ratio po & 96.25 & 20 & 94.38 & 1.87\\
Continuation-ratio npo & 102.43 & 21 & 102.03 & 0.40\\
\hline
\end{tabular}
\end{center}
\end{table}

In Table~\ref{tab:trauma:ori_size}, we list the simulation results with $N=802$, the original sample size.  According to Theorems~\ref{thm:baseline_npo} and \ref{thm:adjacent_npo}, all orders under a baseline-category npo model or an adjacent-categories npo model are indistinguishable in terms of the AIC. Thus, we omit these two models, and list the other six commonly used multinomial logit models. In Table~\ref{tab:trauma:ori_size}, the true order is not evident under any of the multinomial logit models; that is, the rank is not one, or the AIC value is not the smallest. 
To check whether the true order can be regarded approximately as the best one, or whether the difference between the true order and the best order is significant, we denote $\Delta = \text{AIC}_0 - \text{AIC}_*$, the difference in terms of their AIC values. According to \cite{burnham2004aic}, $\Delta\leq 2$ suggests that the true order 
is considered substantially the best; $4\leq \Delta \leq 7$ indicates that the true order is considerably less than optimal; and $\Delta > 10$ suggests the true order is essentially worse than the best one. Because no AIC difference in the last column of  Table~\ref{tab:trauma:ori_size} is greater than two, we conclude that the AIC differences between the true order and the best order are not significant.  Note that because the number of parameters does not change across orders, using the BIC is equivalent to using the AIC here.

The simulation results (not listed in Table~\ref{tab:trauma:ori_size}) also show that all 120 orders under the continuation-ratio po model lead to distinct AIC values. That is, all orders are distinguishable, or no two orders are equivalent, supporting the corresponding statement in Table~\ref{tab:summary}. 

\begin{table}[ht]
\caption{\label{tab:trauma:ten_size}Trauma Simulation Study with $N=8020$}
\begin{center}
\begin{tabular}{c |c c c c}
    \hline
     & \multicolumn{2}{c}{True Order} & \multicolumn{2}{c}{Best Order}\\
     \cline {2-5}
     Model & AIC$_0$ & Rank & AIC$_*$ & AIC$_0-$AIC$_*$\\
     \hline
     Baseline-category po & 133.19 & 1 & 133.19 & 0\\
     Cumulative po & 130.52 & 1 & 130.52 & 0\\
     Cumulative npo & 135.55 & 1 & 135.55 & 0\\
     Adjacent-categories po & 130.31 & 1 & 130.31 & 0\\
     Continuation-ratio po & 130.54 & 1 & 130.54 & 0\\
     Continuation-ratio npo & 141.51 & 21 & 139.96 & 1.55\\
     \hline
    \end{tabular}
\end{center}
\end{table}

In Table~\ref{tab:trauma:ten_size}, we increase the sample size $N$ to $8020$, $10$ times as large as the original one, to numerically check the asymptotic optimality of the true order. That is, the new data $(Y_{i1}', \ldots, Y_{iJ}') \sim {\rm Multinomial}(10 n_i; \pi_{i1}(\hat{\boldsymbol\theta}_o), \ldots,$ $ \pi_{iJ}(\hat{\boldsymbol\theta}_o))$, where $\hat{\boldsymbol\theta}_o$ is fitted from the original data $(Y_{ij})_{ij}$ and $n_i = \sum_{j=1}^J Y_{ij}$~.
Clearly, all models except the continuation-ratio npo model perform best in terms of the true order of the response categories. This confirms our conclusion in Section~\ref{sec:optimal_true_order} that the true order is asymptotically optimal. In this case, the continuation-ratio npo model behaves differently. With $10$ times the original sample size, the true order still ranks 21st, with an even bigger AIC difference 1.55 (not statistically significant either).  
If we further increase the sample size to $N=40\times 802$, the true order ranks third with an AIC difference $0.05$. Actually, even for a very large $N$, there are still six orders among the top tier, the true order, transpositions $(3,4)$ and $(3,5)$, and their equivalent orders (see Theorem~\ref{thm:continuation_npo}). This simulation study provides a numerical example for Case~two described in Section~\ref{sec:optimal_true_order}.

\subsection{True order and treatment effects}\label{sec:order_treatment_effect}

In Section~\ref{sec:optimal_true_order}, we noted that larger absolute values of the regression coefficients may make the true order easier to identify. In this section, we illustrate such a scenario using a cumulative po model, which has fewer coefficients and is easier to modify.

The cumulative po model for the trauma data consists of $p=5$ parameters $\boldsymbol\theta = (\beta_1, \beta_2, \beta_3, \beta_4, \zeta)^T$, where $\zeta$ represents the treatment effect of the dose level $x$. By fitting the model using the original data, we obtain 
the estimated parameters $\hat{\boldsymbol\theta} = (-0.7192, -0.3186, 0.6916, 2.057, -0.1755)^T$. Similarly to Section~\ref{sec:simulation_trauma}, we treat $\sigma_0={\rm id}$ and $\boldsymbol\theta_0 = \hat{\boldsymbol\theta}$ as the true values, and simulate a new data set $(Y'_{ij})_{ij}$ with the original sample size $N=802$. We then check for each possible order $\sigma$ using the simulated data set. In contrast to Section~\ref{sec:simulation_trauma}, we run four simulation studies, with $\zeta = -0.1755, -0.3510, 0.3510, -0.7020$, respectively, to observe how the magnitude of the treatment effects affects the differences between the orders. 

\begin{table}[ht]
\caption{\label{tab:asymp:coef}Best Order and Treatment Effects}
\footnotesize
\begin{center}
\begin{tabular}{c|cc|c|c}
    \hline
    Treatment effect $\zeta$ & AIC (true) & Rank & AIC (best) & AIC (3rd best when true=best) \\
     \hline
     $-0.1755$ & 94.54 & 7 & 93.84 & -\\
     $-0.3510$ & 98.93 & 3 & 98.23 & -\\
     \hspace{0.15cm} $0.3510$ & 90.84 & 1 & 90.84 & 91.61\\
     $-0.7020$ & 85.17 & 1 & 85.17 & 86.58\\
     \hline
    \end{tabular}
\end{center}
\normalsize
\end{table}

From Table~\ref{tab:asymp:coef}, the true order tends to be the best order as $|\zeta|$ increases. For $\zeta=0.3510$ and $-0.7020$, the true order attains the best or minimum AIC value. According to Theorem~\ref{thm:cumulative_order} and its proof, the reversed order with parameters $\boldsymbol\theta_2 = (-\beta_4, -\beta_3, -\beta_2, -\beta_1, -\zeta)^T$ achieves the same AIC value, which can be viewed as the second best result. From $\zeta=0.3510$ to $\zeta=-0.7020$, the AIC difference between the true order and the third best order increases from $91.61 - 90.84= 0.77$ to $86.58 - 85.17 = 1.41$, indicating that a larger treatment effect might make the true order easier to identify.

\subsection{True order and experimental settings}\label{sec:order_x_i}

In this section, we discuss how the  experimental setting ${\mathbf x}_i$ affects the identification of the true order. We use the trauma data under the continuation-ratio npo model as an example, because this is a difficult case in which to identify the true order, according to the results shown in Table~\ref{tab:trauma:ten_size}. 

We explore two ways of changing the set $\{x_i, i=1, \ldots, m\} = \{1,2,3,4\}$ in the trauma data. In the first, we increase the range to $X_A = \{1, 2, \ldots, 16\}$. In the second, we make finer changes to the experiment settings and obtain $X_B = \{1, 1.25, 1.50,$ $\ldots,$ $3.75, 4\}$, the range of which is still $[1,4]$. In both methods, the number $m$ of experimental settings increases significantly.

Similarly to Section~\ref{sec:simulation_trauma}, we treat $\sigma_0 = {\rm id}$ and $\boldsymbol\theta_0 = \hat{\boldsymbol\theta}$, estimated for the continuation-ratio npo model from the original data, as the true values. In this section, we simulate $100$ new data sets $(Y_{ij}^{(b)})_{ij}$ independently, with $N=100\times 210m$ and $b=1, \ldots, 100$, using the continuation-ratio npo model with $\sigma_0$ and $\boldsymbol\theta_0$. For each $b=1, \ldots, 100$ and each $\sigma \in {\mathcal P}$, we fit the continuation-ratio npo model using the permuted data $(Y_{i\sigma(j)}^{(b)})_{ij}$ against the dose level $x_i$. The corresponding AIC values are denoted as AIC$_\sigma^{(b)}$, for $b=1, \ldots, 100$. To compare the true order $\sigma_0$ with each $\sigma$ of the other $119$ orders, we run a one-sided paired $t$-test on $({\rm AIC}_{\sigma_0}^{(b)})_b$ and $({\rm AIC}_{\sigma}^{(b)})_b$. A significant $p$-value indicates that the AIC value associated with $\sigma_0$ is significantly smaller than the AIC value of $\sigma$.

Under the first scenario $X_A$,  $118$ out of $119$ p-values are statistically significant, indicating that the true order is significantly better than all other orders, except the equivalent one listed by Theorem~\ref{thm:continuation_npo}. That is, an increased range may make the true order easier to identify. 

Under the second scenario $X_B$, there are still three $p$-values greater than $0.05$. Further tests indicate that these four orders, including the true order, are indistinguishable. That is, increasing $m$ while maintaining the range of $x_i$ may not be an efficient way to improve the identifiability of the true order.

\subsection{Choice of baseline category}\label{sec:baseline_category}

In this section, we use cross-validation to show that the choice of baseline category makes a difference for baseline-category po models.

\begin{table}[ht]
\caption{\label{tab:basepocv_data}Data from Baseline-category po Model with Baseline Category $J=4$}
\centering
\begin{tabular}{c|c c c c|c}
    \hline
     $x_i=i$ & $Y_{i1}$ & $Y_{i2}$ & $Y_{i3}$ & $Y_{i4}$ & $n_i$ \\
     \hline
     1 & 22 & 33 & 10 & 35 & 100\\
     2 & 31 & 40 & 14 & 15 & 100\\
     3 & 23 & 43 & 22 & 12 & 100\\
     4 & 27 & 49 & 18 &  6 & 100\\
     \hline
    \end{tabular}
\end{table}

Table~\ref{tab:basepocv_data} provides  simulated data from a baseline-category po model, with the fourth category as the true baseline category. The parameters used for simulating the data are $\boldsymbol{\theta}=(\beta_1,\beta_2,\beta_3,\zeta)^T=(-0.8,-0.3,-1.0,0.5)^T$.
There are $n=400$ observations in total. We randomly split the observations into two parts, $267$ as training data and $133$ as testing data. We repeat the random partition $100$ times. For each random partition and each order $\sigma\in {\mathcal P}$, (1) we use the training data to fit the baseline-category po model with order $\sigma$ (similarly to Section~\ref{sec:simulation_trauma}), and denote the fitted model as Model~$\sigma$; (2) we predict the labels of the responses of the testing data using Model~$\sigma$, and use the cross-entropy loss (e.g., see \cite{hastie2009elements}) to measure the prediction errors.
In this way, we have $100$ prediction errors (cross-entropy loss) for each $\sigma$. For any two orders, we can run a one-sided paired $t$-test to check whether one order's prediction error is significantly lower than that of the other (similarly to Section~\ref{sec:order_x_i}). 

We conclude the following from this cross-validation study: (i) all orders that share the same baseline category have the same cross-entropy loss, indicating that they are indistinguishable in terms of their prediction errors, and supporting the results of Theorem~\ref{thm:baseline_ppo}; (ii) supported by the pairwise $t$-tests, the orders with the true baseline ($J=4$) have significantly smaller cross-entropy losses than the other orders, with $p$-values $5.69 \times 10^{-42}, 9.54 \times 10^{-47}$, and $8.99 \times 10^{-51}$, respectively, showing that the correct choice of baseline category matters in practice.

\subsection{When true order does not exist}\label{sec:order_nominal}

In this section, we investigate an order misspecification issue when the true order does not exist. More specifically, for a baseline-category npo model, all orders are equivalent, according to Theorem~\ref{thm:baseline_npo}. In other words, there is no true order. Nevertheless, given such a data set, we can still find the best model with the best order, called the {\it working} order. 
The simulation study below shows that when the true order does not exist, we have the following results: (1) with moderate sample sizes, a working order with a different model can be selected, but may not be significantly better than the true model in terms of the AIC; and (2) asymptotically, the true model without a true order may be significantly better than any other model with any working order. 

In this simulation study, we use the baseline-category npo model fitted from the original trauma data as the true model, and simulate a data set with the set of covariate levels $X_A=\{1, 2, \ldots, 16\}$, which is a more informative experimental setting (see Section~\ref{sec:order_x_i}). For each level $x_i \in X_A$, we simulate $n_A = 200$ observations, with the total sample size $N = 16 n_A = 3,200$. For the simulated data with $N = 3,200$, the best model according to the AIC is a continuation-ratio npo model, with AIC = $321.31$, at its best (working) order $\{${\tt death}, {\tt major disability}, {\tt vegetative state}, {\tt minor disability}, {\tt good recovery}$\}$, whereas the AIC value of the true model is $322.48$. In other words, the best model with the working order is not significantly better than the true model. If we increase the sample size $N$ to $3,200\times 100$, the continuation-ratio npo model is still the best model, with AIC = $604.93$ at the best order $\{${\tt vegetative state}, {\tt death}, {\tt minor disability}, {\tt major disability}, {\tt good recovery}$\}$, which again is not significantly better than the true model, with AIC = $606.75$. If we further increase $N$ to $3,200 \times 10,000$, the true model becomes the best model, with AIC = $914.80$, and is significantly better than the continuation-ratio npo model, with AIC = $944.92$ at its best order.

We repeat the procedure $100$ times with various $N$. For each simulated trauma data set from the baseline-catetory npo model, we calculate the difference between  AIC$_{\rm true}$ (the AIC value with the true model) and AIC$_{\rm other}$ (the smallest AIC value among all other models and all orders). In Table~\ref{tab:frequency_tale}, we list the frequencies of the AIC differences falling into different ranges. For example, The first column ``$<0$'' provides the numbers of simulations out of 100 with AIC$_{\rm true}$ $-$ AIC$_{\rm other} < 0$. As $N$ increases from $1\times 3,200$ to $10,000 \times 3,200$, the number of cases increases from $0$ to $100$, showing that the true model is increasingly likely to outperform other models with any order in terms of AIC values. The other columns in Table~\ref{tab:frequency_tale} show similar patterns, confirming the conclusions described at the beginning of this section.

\begin{table}[ht]
\caption{\label{tab:frequency_tale}Frequencies of AIC$_{\rm true}$ $-$ AIC$_{\rm other}$ Categories out of $100$ Simulated Trauma Data from Baseline-category npo Model}
\centering
\begin{center}
\begin{tabular}{r|rrrrrr}
    \hline
$N$ &   $<0$ &   $[0, 2)$  &  $[2, 4)$  & $[4, 7)$ & $[7, 10)$  & $\geq 10$\\ \hline
$1\times 3,200$ &    0  &    57  &   27  &   11  &   2   &      3\\
$10\times 3,200$   & 13  &  55 &    25  &   7  &     0  &        0\\
$100\times 3,200$  &  36  &  44 &    16  &   3    &    0   &       1\\
$10,000\times 3,200$   &   100   &  0   &    0   &   0    &    0     &     0
\\
    \hline
    \end{tabular}
\end{center}
\end{table}

\section{Real-Data Analysis}\label{sec:real_data}

\begin{table}[ht]
\caption{\label{tab:police_data}US Police Involved Fatalities Data (2000--2016)}
\centering
\footnotesize
\begin{center}
\begin{tabular}{cccc|rrrr|r}
    \hline
    Armed & Gender & Flee & Mental Illness & Other & Shot & Shot and Tasered & Tasered & Total \\
    \hline
    Gun & Female & False & False & 0 & 134 & 0 & 1 & 135\\
    Gun & Female & False & True & 0 & 57 & 0 & 1 & 58\\
    Gun & Female & True & False & 0 & 8 & 0 & 0 & 8\\
    Gun & Female & True & True & 0 & 4 & 0 & 0 & 4\\
    Gun & Male & False & False & 2 & 3314 & 6 & 35 & 3357\\
    Gun & Male & False & True & 0 & 810 & 4 & 15 & 829\\
    Gun & Male & True & False & 0 & 271 & 5 & 0 & 276\\
    Gun & Male & True & True & 0 & 33 & 1 & 0 & 36\\
    Other & Female & False & False & 0 & 53 & 1 & 0 & 54\\
    Other & Female & False & True & 1 & 42 & 1 & 0 & 44\\
    Other & Female & True & False & 0 & 4 & 1 & 0 & 5\\
    Other & Female & True & True & 0 & 2 & 0 & 0 & 2\\
    Other & Male & False & False & 1 & 910 & 38 & 10 & 959\\
    Other & Male & False & True & 2 & 478 & 21 & 5 & 506\\
    Other & Male & True & False & 0 & 114 & 10 & 0 & 124\\
    Other & Male & True & True & 0 & 14 & 2 & 0 & 16\\
    Unarmed & Female & False & False & 1 & 231 & 0 & 3 & 235\\
    Unarmed & Female & False & True & 0 & 61 & 0 & 5 & 66\\
    Unarmed & Female & True & False & 0 & 2 & 0 & 0 & 2\\
    Unarmed & Male & False & False & 10 & 4338 & 16 & 253 & 4617\\
    Unarmed & Male & False & True & 12 & 832 & 5 & 214 & 1063\\
    Unarmed & Male & True & False & 0 & 75 & 8 & 0 & 83\\
    Unarmed & Male & True & True & 0 & 5 & 1 & 0 & 6\\
    \hline
    \end{tabular}
{\footnotesize Note: The group (Unarmed, Female, Fleed, Has Mental Illness) contains no observation and is omitted.} 
\end{center}
\normalsize
\end{table}

The US Police Involved Fatalities Data (hereafter, Police data) were downloaded from data.world (\url{https://data.world/awram/us-police-involved-fatalities}, version June 21, 2020), which was collected by Chris Awarm from three data resources, namely, \url{https://fatalencounters.org/}, \url{https://www.gunviolencearchive.org/}, and {\tt Fatal Police Shootings}, from data.world.
The original data lists individuals killed by the police in the United States from 2000 to 2016, including information on 12,483 suspects' age, race, mental health status, weapons they were armed with, and whether or not they were fleeing. By way of example, we focus on whether the police's action can be predicted by the aforementioned information related to a suspect.

As summarized in Table~\ref{tab:police_data}, there are four categories of (police) responses, namely, {\tt other}, {\tt shot}, {\tt shot and tasered}, and {\tt tasered}. In our notation, $J=4$.
In the original data, there are $60$ different types of armed status. Here, we simplify these into three categories: {\tt gun}, if the original input is ``gun,''  {\tt unarmed}, if the original input is ``unarmed'' or missing, and {\tt other}, if otherwise. As such, we have $24$ possible level combinations of \texttt{armed status} ($x_{i1} = $ $1$ (gun), $2$ (other), or $3$ (unarmed)), \texttt{gender} ($x_{i2}=$ $0$ (female) or $1$ (male)), \texttt{flee} ($x_{i3}=$ $0$ (false) or $1$ (true)), \texttt{mental illness} ($x_{i4}=$ $0$ (false) or $1$ (true)). Because there is no observation associated with ${\mathbf x}_i = (x_{i1}, x_{i2}, x_{i3}, x_{i4})^T = (3, 0, 1, 1)^T$, $m=23$ in this case (see Table~\ref{tab:police_data}).

As an example, we consider the main-effects baseline-category, cumulative, adjacent-categories, and continuation-ratio logit models with po or npo. In our notation, ${\mathbf h}_1({\mathbf x}_i) = \cdots = {\mathbf h}_{J-1}({\mathbf x}_i) = (1, {\mathbf 1}_{\{x_{i1}=2\}}, {\mathbf 1}_{\{x_{i1}=3\}}, x_{i2}, x_{i3},$ $ x_{i4})^T$ for all eight models under consideration. For each model, we choose the best order out of $4!=24$ of the four response categories, based on the AIC. Of the eight logit models, each with $24$ orders, the continuation-ratio npo model with the chosen order (t, s, o, st) or (t, s, st, o) (see Table~\ref{tab:police_order}) performs best, and can be written as 
\begin{equation*}
    \log\left(\frac{\pi_{ij}}{\pi_{i,j+1}+\cdots+\pi_{iJ}}\right)=\beta_{j1}+\beta_{j2}{\mathbf 1}_{\{x_{i1}=2\}}+\beta_{j3}{\mathbf 1}_{\{x_{i1}=3\}}+\beta_{j4}x_{i2}+\beta_{j5}x_{i3}+\beta_{j6}x_{i4} \ ,
\end{equation*}
where $j=1,2,3$ and $i=1,\ldots,23$. The corresponding BIC values, not shown here, provide a consistent selection result. According to the AIC values, if we choose the continuation-ratio npo model with the best order $\{${\tt tasered}, {\tt shot}, {\tt other}, {\tt shot and tasered}$\}$ against the baseline-category models (po or npo), which are commonly used for categorical responses without an order, the improvement in the prediction accuracy is significant (AIC differences $>5$). Note that in this case, it is not trivial to determine the baseline category for po models, owing to the existence of the category {\tt other}.

Table~\ref{tab:police_order} also shows large gaps of AIC values between the npo models and the corresponding po models, indicating that npo models are significantly better than the corresponding po models in this case. The differences within the AIC values of the npo models are also much smaller than those within the po models. This provides strong evidence that for the Police data, the parameters for the categories are very different, and thus the proportional odds (po) assumptions are not appropriate (see the Introduction). 

To validate the selected model, we conduct  five-fold cross-validation for the data, with cross-entropy loss as the criterion (e.g., see \cite{hastie2009elements}). We compare our selected continuation-ratio npo model with the baseline-category npo model, which is commonly used for nominal responses. In terms of cross-entropy loss, the continuation-ratio npo model achieves $550.50$, which is less than the value of $555.91$ for the baseline-category npo model. This is consistent with our conclusion based on AIC values.

The estimated parameters for the chosen continuation-ratio npo model are provided in Table~\ref{tab:police:result}, which can be used to interpret the roles and effects of different factors. For example, $\hat\beta_{13} = 2.03$ indicates that the estimated odds ratio of {\tt tasered} and ``unarmed'' is $e^{2.03}=7.61$, which implies that ``unarmed'' leads to a much smaller chance of being shot. 
In contrast, $\hat\beta_{15} = -18.02$, with an estimated odds ratio $e^{-18.02}=1.49\times 10^{-8}$, implies that suspects who flee have a much greater chance of being shot. Because {\tt shot} is usually regarded as more severe than {\tt tasered}, the estimated parameters imply that if suspects show a greater threat such as being armed (gun or other), or try to flee, the police tend to take more extreme actions, such as shooting a gun.

\begin{table}[t!]
\caption{\label{tab:police_order}Model and Order Selection for Police Data}
\begin{center}
\begin{tabular}{c|c| c}
    \hline
     Model & AIC with Best Order & Best Order\\
     \hline
     Baseline-category po & 401.33 & t is the baseline\\
     Baseline-category npo & 197.81 & All are the same\\
     Cumulative po & 318.14 & (st, s, o, t) or (t, o, s, st)\\
     Cumulative npo & 194.48 & (o, st, s, t) or (t, s, st, o)\\
     Adjacent-categories po & 290.17 & (st, s, o, t) or (t, o, s, st)\\
     Adjacent-categories npo & 197.81 & All are the same\\
     Continuation-ratio po & 320.22 & (t, o, s, st)\\
     Continuation-ratio npo & 192.01 & (t, s, o, st) or (t, s, st, o)\\
    \hline
    \multicolumn{3}{l}{\small Note: s = {\tt shot}, t = {\tt tasered}, o = {\tt other}, st = {\tt shot and tasered}.}
    \end{tabular}
\end{center}
\end{table}

\begin{table}[t!]
\caption{\label{tab:police:result}Estimated Parameters for Police Data under Continuation-ratio npo Model}
\centering
\footnotesize
\begin{tabular}{r|r r r r r r}
    \hline
     $j$ & Intercept & Armed Status & Armed Status & Gender & Flee & Mental Illness \\
     & $\hat\beta_{j1}$ & $\hat\beta_{j2}$ (Other) & $\hat\beta_{j3}$ (Unarmed) & $\hat\beta_{j4}$ & $\hat\beta_{j5}$ & $\hat\beta_{j6}$\\
     \hline
     1 & -6.00 & -0.44 & 2.03 & 1.17 & -18.02 & 1.34\\
     2 & 6.47 & -2.43 & -1.09 & -0.58 & -1.55 & -0.59\\
     3 & 0.26 & -1.49 & 1.69 & -2.29 & -28.35 & 1.02\\
     \hline
         \multicolumn{7}{l}{\small Note: $j=1,2,3$ correspond to {\tt tasered}, {\tt shot}, and {\tt other}, respectively.}
    \end{tabular}
\normalsize    
\end{table}

To investigate whether the best order is chosen because of randomness, we conduct a simulation study similar to that in Section~\ref{sec:order_nominal}. We regenerate the Police data simulated from the baseline-category npo model fitted from the original data (see Table~\ref{tab:police:basenpo_result} for the parameter values). For the simulated Police data, the best model is the cumulative npo model with the order $\{${\tt shot and tasered}, {\tt shot}, {\tt other}, {\tt tasered}$\}$ with AIC $149.77$, which is not significantly better than the true model (baseline-category npo), with AIC $151.04$. 
If we increase the sample size by a factor of $10$, the true model with AIC $271.90$ becomes better than the cumulative npo model, with AIC $274.53$. If we further increase it by a factor of $100$, the true model with AIC $401.92$ is significantly better than the cumulative npo model, with AIC $440.45$. 
In other words, with the original sample size of the Police data, if there is no true order, then the selected model with a working order may not be significantly better than the true model. As the amount of data increases, the true model becomes significantly better than the other models with any working order, supporting the simulation results in Section~\ref{sec:order_nominal}.

\begin{table}[t!]
\caption{\label{tab:police:basenpo_result}Estimated Parameters for Police Data under Baseline-category npo Model}
\centering
\footnotesize
\begin{tabular}{r|r r r r r r}
    \hline
     $j$ & Intercept & Armed Status & Armed Status & Gender & Flee & Mental Illness \\
     & $\hat\beta_{j1}$ & $\hat\beta_{j2}$ (Other) & $\hat\beta_{j3}$ (Unarmed) & $\hat\beta_{j4}$ & $\hat\beta_{j5}$ & $\hat\beta_{j6}$\\
     \hline
     1 & -2.07 & -1.34 & 6.00 & -1.18 & -0.82 & -0.31\\
     2 & 0.24 & -8.83 & -2.04 & 11.09 & -1.39 & 12.74\\
     3 & 1.89 & 0.19 & 0.40 & -1.35 & 2.93 & -1.02\\
     \hline
         \multicolumn{7}{l}{\small Note: $j=1,2,3$ correspond to {\tt other}, {\tt shot}, {\tt shot and tasered}, respectively.}    \end{tabular}
\normalsize
\end{table}

\section{Discussion}\label{sec:discussion}

In our analysis of real data, we consider both the type of model and the order of the categories as parts of the model selection procedure. Our examples show that the improvement by choosing a more appropriate order can be highly significant. For example, if we focus on adjacent-categories po models for the Police data, the smallest AIC value from the best order is $290.17$, whereas the largest one from the worst order is $674.44$, implying a significant difference in terms of prediction accuracy.

For the Police data, there is no natural order, owing to the existence of the category {\tt other}. According to the AIC values, the continuation-ratio npo model with the working order (t, s, o, st), or equivalently (t, s, st, o), is significantly better than the baseline-category npo models (see Table~\ref{tab:police_order}). On the one hand, if there is a true order, it will appear among the best orders consistently (see Section~\ref{sec:optimal_true_order}). On the other hand, if there is no true order for the Police data, then there is no ordinal model with a working order that is significantly better than the baseline-category npo model, with the current or a larger sample size (see the simulation study at the end of Section~\ref{sec:real_data}). In other words, we can report the working order with confidence if it is significantly better than other orders.

Finally, in the proofs of Theorems~\ref{thm:baseline_ppo}, \ref{thm:cumulative_order}, \ref{thm:adjacent_ppo}, \ref{thm:baseline_npo}, \ref{thm:adjacent_npo}, and \ref{thm:continuation_npo}, we provide explicit transformation formulae from $\boldsymbol\theta_1$ with $\sigma_1$ to $\boldsymbol\theta_2$ with $\sigma_2$ when $\sigma_2 \sim \sigma_1$. This significantly reduces the computational cost of finding MLEs with different orders.

\bigskip
\noindent {\bf Supplementary Material}

\noindent The online Supplementary Material contains proofs of Theorems~\ref{thm:permutation}, \ref{thm:baseline_ppo}, \ref{thm:cumulative_order}, \ref{thm:adjacent_ppo}, \ref{thm:baseline_npo}, \ref{thm:adjacent_npo}, \ref{thm:continuation_npo}, and \ref{thm:consist_order}, Lemmas~\ref{lem:mle_pi_ij} and \ref{lem:limit_l_N}, and Corollary~\ref{cor:AIC_BIC}.

\bigskip
\noindent {\bf Acknowledgements}

\noindent This work was supported, in part, by the US NSF grant DMS-1924859.

\bibhang=1.7pc
\bibsep=2pt
\fontsize{9}{14pt plus.8pt minus .6pt}\selectfont
\renewcommand\bibname{\large \bf References}
\expandafter\ifx\csname
natexlab\endcsname\relax\def\natexlab#1{#1}\fi
\expandafter\ifx\csname url\endcsname\relax
  \def\url#1{\texttt{#1}}\fi
\expandafter\ifx\csname urlprefix\endcsname\relax\def\urlprefix{URL}\fi

\clearpage
\setcounter{page}{1}
\def\thepage{S\arabic{page}}

\centerline{\large\bf Identifying the Most Appropriate Order}
\vspace{2pt}
 \centerline{\large\bf for Categorical Responses}
\vspace{.25cm}
 \centerline{Tianmeng Wang and Jie Yang} 
\vspace{.4cm}
 \centerline{\it University of Illinois at Chicago}
\vspace{.55cm}
 \centerline{\bf Supplementary Material}
\vspace{.55cm}
\fontsize{9}{11.5pt plus.8pt minus .6pt}\selectfont
\par

\setcounter{section}{0}
\setcounter{equation}{0}
\def\theequation{S.\arabic{equation}}
\def\thesection{S\arabic{section}}

\fontsize{12}{14pt plus.8pt minus .6pt}\selectfont


\begin{proof} of Theorem~2.1:
The log-likelihood of the model at $\boldsymbol\theta_1$ with permuted responses ${\mathbf Y}_i^{\sigma_1}$ is
\begin{equation*}
    l_1(\boldsymbol\theta_1)=\sum_{i=1}^m \log(n_i!)-\sum_{i=1}^m \sum_{j=1}^J \log(Y_{i\sigma_1(j)}!)+\sum_{i=1}^m \sum_{j=1}^J Y_{i\sigma_1(j)}\log \pi_{ij}(\boldsymbol\theta_1)
\end{equation*}
while the log-likelihood at $\boldsymbol\theta_2$ with ${\mathbf Y}_i^{\sigma_2}$ is
\begin{equation*}
    l_2(\boldsymbol\theta_2)=\sum_{i=1}^m \log(n_i!)-\sum_{i=1}^m \sum_{j=1}^J \log(Y_{i\sigma_2(j)}!)+\sum_{i=1}^m \sum_{j=1}^J Y_{i\sigma_2(j)}\log \pi_{ij}(\boldsymbol\theta_2)
\end{equation*}
Since ${\mathbf Y}_i^{\sigma_1}$ and ${\mathbf Y}_i^{\sigma_2}$ are different only at the order of individual terms, $$\sum_{j=1}^J \log(Y_{i\sigma_1(j)}!) = \sum_{j=1}^J \log(Y_{i\sigma_2(j)}!)$$ 
On the other hand, $\pi_{ij}(\boldsymbol\theta_2) = \pi_{i\sigma_2(\sigma_1^{-1}(j))}(\boldsymbol\theta_1)$ implies that \begin{eqnarray*}
\sum_{j=1}^J Y_{i\sigma_2(j)} \log \pi_{ij}(\boldsymbol\theta_2) &=& \sum_{j=1}^J Y_{i\sigma_2(j)} \log \pi_{i\sigma_2(\sigma_1^{-1}(j))}(\boldsymbol\theta_1)\\ &=& \sum_{j=1}^J Y_{ij} \log \pi_{i\sigma_1^{-1}(j)}(\boldsymbol\theta_1)\\ &=& \sum_{j=1}^J Y_{i\sigma_1(j)} \log \pi_{ij}(\boldsymbol\theta_1)
\end{eqnarray*}
Therefore, $l_1(\boldsymbol\theta_1) = l_2(\boldsymbol\theta_2)$, which implies $\max_{\boldsymbol\theta_1} l_1(\boldsymbol\theta_1) \leq \max_{\boldsymbol\theta_2} l_2 (\boldsymbol\theta_2)$. Similarly, $\max_{\boldsymbol\theta_2} l_2 (\boldsymbol\theta_2) \leq \max_{\boldsymbol\theta_1} l_1(\boldsymbol\theta_1)$. Thus $\max_{\boldsymbol\theta_1} l_1(\boldsymbol\theta_1) = \max_{\boldsymbol\theta_2} l_2 (\boldsymbol\theta_2)$.

Given that (2.4) is true for $\sigma_1$ and $\sigma_2$, for any permutation $\sigma \in {\mathcal P}$, $$\pi_{i\sigma_1^{-1}(\sigma^{-1}(j))} (\boldsymbol\theta_1) = \pi_{i\sigma_2^{-1}(\sigma^{-1}(j))} (\boldsymbol\theta_2)$$ for all $i$ and $j$. That is, $\pi_{i(\sigma\sigma_1)^{-1}(j)} (\boldsymbol\theta_1) = \pi_{i(\sigma\sigma_2)^{-1}(j)} (\boldsymbol\theta_2)$ for all $i$ and $j$. Following the same proof above, we have $\sigma\sigma_1 \sim \sigma\sigma_2$~.
\end{proof}

\begin{proof} of Theorem~2.2: 
We first show that $\sigma_1={\rm id}$, the identity permutation, and any permutation $\sigma_2$ satisfying $\sigma_2(J)=J$ are equivalent. Actually, given any $\boldsymbol\theta_1 = (\boldsymbol\beta_1^T, \ldots, \boldsymbol\beta_{J-1}^T, \boldsymbol\zeta^T)^T$ for $\sigma_1 = {\rm id}$, we let $\boldsymbol\theta_2 = (\boldsymbol\beta_{\sigma_2(1)}^T, \ldots, \boldsymbol\beta_{\sigma_2(J-1)}^T, \boldsymbol\zeta^T)^T$ for $\sigma_2$. Then $\eta_{ij}(\boldsymbol\theta_2) = {\mathbf h}_j^T({\mathbf x}_i) \boldsymbol\beta_{\sigma_2(j)} + {\mathbf h}_c^T({\mathbf x}_i) \boldsymbol\zeta = {\mathbf h}_{\sigma_2(j)}^T({\mathbf x}_i) \boldsymbol\beta_{\sigma_2(j)} + {\mathbf h}_c^T({\mathbf x}_i) \boldsymbol\zeta = \eta_{i\sigma_2(j)} (\boldsymbol\theta_1)$ for all $i=1, \ldots, m$ and $j=1, \ldots, J-1$. According to (2.2) and (2.3), $\pi_{ij}(\boldsymbol\theta_2) = \pi_{i\sigma_2(j)} (\boldsymbol\theta_1)$ for all $i=1, \ldots, m$ and $j=1, \ldots, J$. Then ${\rm id} \sim \sigma_2$ is obtained by Theorem~2.1. 

For general $\sigma_1$ and $\sigma_2$ satisfying $\sigma_1(J) = \sigma_2(J) = J$, ${\rm id} \sim \sigma_1^{-1}\sigma_2$ implies $\sigma_1 \sim \sigma_2$ by Theorem~2.1.
\end{proof}

\begin{proof} of Theorem~2.3:
Given $\boldsymbol\theta_1 = (\boldsymbol\beta_1^T, \boldsymbol\beta_2^T, \ldots, \boldsymbol\beta_{J-1}^T, \boldsymbol\zeta^T)^T$ with $\sigma_1$, we let $\boldsymbol\theta_2 = (-\boldsymbol\beta_{J-1}^T, -\boldsymbol\beta_{J-2}^T, \ldots, -\boldsymbol\beta_1^T, -\boldsymbol\zeta^T)^T$ for $\sigma_2$. Then $\eta_{ij}(\boldsymbol\theta_2) = -{\mathbf h}_j^T({\mathbf x}_i) \boldsymbol\beta_{J-j} - {\mathbf h}_c^T ({\mathbf x}_i) \boldsymbol\zeta = -{\mathbf h}_{J-j}^T({\mathbf x}_i) \boldsymbol\beta_{J-j} - {\mathbf h}_c^T ({\mathbf x}_i) \boldsymbol\zeta = -\eta_{i,J-j}(\boldsymbol\theta_1)$ and thus $\rho_{ij}(\boldsymbol\theta_2) = 1-\rho_{i,J-j}(\boldsymbol\theta_1)$ for all $i=1, \ldots, m$ and $j=1, \ldots, J-1$. It can be verified that $\pi_{ij}(\boldsymbol\theta_2) = \pi_{i,J+1-j}(\boldsymbol\theta_1)$ for all $i=1, \ldots, m$ and $j=1, \ldots, J$ according to (2.2) and (2.3). Then for all $i=1, \ldots, m$ and $j=1, \ldots, J$,
\begin{eqnarray*}
\pi_{i\sigma_2^{-1}(j)}(\boldsymbol\theta_2) &=& \pi_{i,J+1-\sigma_2^{-1}(j)}(\boldsymbol\theta_1)
= \pi_{i\sigma_1^{-1}(\sigma_1(J+1-\sigma_2^{-1}(j)))} (\boldsymbol\theta_1)\\
&=& \pi_{i\sigma_1^{-1}(\sigma_2(\sigma_2^{-1}(j)))} (\boldsymbol\theta_1)
= \pi_{i\sigma_1^{-1}(j)} (\boldsymbol\theta_1)
\end{eqnarray*}
That is, (2.4) holds given $\boldsymbol\theta_1$~. Since it is one-to-one from $\boldsymbol\theta_1$ to $\boldsymbol\theta_2$, (2.4) holds given $\boldsymbol\theta_2$ as well. According to Theorem~2.1, $\sigma_1 \sim \sigma_2$~.
\end{proof}

\begin{proof} of Theorem~2.4:
Similar as the proof of Theorem~2.3, for ppo models satisfying ${\mathbf h}_1({\mathbf x}_i) = \cdots ={\mathbf h}_{J-1} ({\mathbf x}_i)$, given $\boldsymbol\theta_1 = (\boldsymbol\beta_1^T, \boldsymbol\beta_2^T, \ldots, \boldsymbol\beta_{J-1}^T, \boldsymbol\zeta^T)^T$ with $\sigma_1$, we let $\boldsymbol\theta_2 = (-\boldsymbol\beta_{J-1}^T, -\boldsymbol\beta_{J-2}^T, \ldots, -\boldsymbol\beta_1^T, -\boldsymbol\zeta^T)^T$ for $\sigma_2$~. Then $\eta_{ij}(\boldsymbol\theta_2) = -{\mathbf h}_j^T({\mathbf x}_i) \boldsymbol\beta_{J-j} - {\mathbf h}_c^T ({\mathbf x}_i) \boldsymbol\zeta = -{\mathbf h}_{J-j}^T({\mathbf x}_i) \boldsymbol\beta_{J-j} - {\mathbf h}_c^T ({\mathbf x}_i) \boldsymbol\zeta = -\eta_{i,J-j}(\boldsymbol\theta_1)$ and thus $\rho_{ij}(\boldsymbol\theta_2) = 1-\rho_{i,J-j}(\boldsymbol\theta_1)$ for all $i=1, \ldots, m$ and $j=1, \ldots, J-1$. It can be verified that for $j=1, \ldots, J-1$,
\[
\prod_{l=j}^{J-1} \frac{\rho_{il}(\boldsymbol\theta_2)}{1-\rho_{il}(\boldsymbol\theta_2)} = \prod_{l=j}^{J-1} \frac{1-\rho_{i,J-l}(\boldsymbol\theta_1)}{\rho_{i,J-l}(\boldsymbol\theta_1)} = \prod_{l=1}^{J-j} \frac{1-\rho_{il}(\boldsymbol\theta_1)}{\rho_{il}(\boldsymbol\theta_1)}
\]
implies $\pi_{ij}(\boldsymbol\theta_2) = \pi_{i,J+1-j}(\boldsymbol\theta_1)$ for all $i=1, \ldots, m$ and $j=1, \ldots, J$ due to (2.2) and (2.3). Then $\pi_{i\sigma_2^{-1}(j)}(\boldsymbol\theta_2) = \pi_{i\sigma_1^{-1}(j)} (\boldsymbol\theta_1)$ similarly as in the proof of Theorem~2.3, which leads to $\sigma_1 \sim \sigma_2$ based on Theorem~2.1.
\end{proof}

\begin{proof} of Theorem~2.5: 
According to Theorem~2.1, we only need to show that (2.4) holds for $\sigma_1 = {\rm id}$ and an arbitrary permutation $\sigma_2 \in {\mathcal P}$.

{\it Case one:} $\sigma_2(J)=J$. In this case, for any $\boldsymbol\theta_1 = (\boldsymbol\beta_{1}^T, \ldots, \boldsymbol\beta_{J-1}^T)^T$, we let  $\boldsymbol\theta_2 = (\boldsymbol\beta_{\sigma_2(1)}^T, \ldots, \boldsymbol\beta_{\sigma_2(J-1)}^T)^T$. Then $\eta_{ij}(\boldsymbol\theta_2) = {\mathbf h}_j^T({\mathbf x}_i) \boldsymbol\beta_{\sigma_2(j)} = {\mathbf h}_{\sigma(j)}^T({\mathbf x}_i) \boldsymbol\beta_{\sigma_2(j)}$ $=$ $\eta_{i\sigma_2(j)}(\boldsymbol\theta_1)$, which leads to $\rho_{ij}(\boldsymbol\theta_2) = \rho_{i\sigma_2(j)} (\boldsymbol\theta_1)$ for all $i=1, \ldots, m$ and $j=1, \ldots, J-1$. According to (2.2) in Lemma~1, $\pi_{ij}(\boldsymbol\theta_1) = \pi_{i\sigma_2^{-1}(j)}(\boldsymbol\theta_2)$, which is (2.4) in this case. Since the correspondence between $\boldsymbol\theta_1$ and $\boldsymbol\theta_2$ is one-to-one, then (2.4) holds for given $\boldsymbol\theta_2$ as well.

{\it Case two:} $\sigma_2(J)\neq J$. Given $\boldsymbol\theta_1 = (\boldsymbol\beta_{1}^T, \ldots, \boldsymbol\beta_{J-1}^T)^T$, we let  $\boldsymbol\theta_2 = (\boldsymbol\beta_{21}^T, \ldots,$ $\boldsymbol\beta_{2,J-1}^T)^T$ such that for $j=1, \ldots, J-1$,
\[
\boldsymbol\beta_{2j} = \left\{
\begin{array}{cl}
\boldsymbol\beta_{\sigma_2(j)} - \boldsymbol\beta_{\sigma_2(J)} & \mbox{ if }j\neq \sigma_2^{-1}(J)\\
-\boldsymbol\beta_{\sigma_2(J)} & \mbox{ if } j = \sigma_2^{-1}(J)
\end{array}\right.
\]
Then for $i=1, \ldots, m$ and $j=1, \ldots, J-1$, 
\[
\eta_{ij}(\boldsymbol\theta_2) = {\mathbf h}_j^T({\mathbf x}_i) \boldsymbol\beta_{2j} = \left\{
\begin{array}{cl}
\eta_{i\sigma_2(j)} (\boldsymbol\theta_1) - \eta_{i\sigma_2(J)} (\boldsymbol\theta_1) & \mbox{ if }j\neq \sigma_2^{-1}(J)\\
-\eta_{i\sigma_2(J)} (\boldsymbol\theta_1) & \mbox{ if } j = \sigma_2^{-1}(J)
\end{array}\right.
\]
It can be verified that: (i) If
$\sigma_2^{-1}(j) \neq J$ and $j\neq J$, then $\pi_{i\sigma_2^{-1}(j)}(\boldsymbol\theta_2) = \pi_{ij}(\boldsymbol\theta_1)$ according to (2.2); (ii) $\pi_{i\sigma_2^{-1}(J)}(\boldsymbol\theta_2) = \pi_{iJ}(\boldsymbol\theta_1)$ according to (2.2); and (iii) if
$\sigma_2^{-1}(j) = J$, then $\pi_{i\sigma_2^{-1}(j)}(\boldsymbol\theta_2) = \pi_{ij}(\boldsymbol\theta_1)$ according to (2.3). Thus (2.4) holds given $\boldsymbol\theta_1$~. Given $\sigma_2$, the correspondence between $\boldsymbol\theta_1$ and $\boldsymbol\theta_2$ is one-to-one, then (2.4) holds given $\boldsymbol\theta_2$ as well. Thus ${\rm id}\sim \sigma_2$ according to Theorem~2.1.

For general $\sigma_1$ and $\sigma_2$,  ${\rm id} \sim \sigma_1^{-1}\sigma_2$ implies $\sigma_1\sim \sigma_2$ according to Theorem~2.1.
\end{proof}

\begin{proof} of Theorem~2.6: 
Similar as the proof of Theorem~2.5, we first show that (2.4) holds for $\sigma_1 = {\rm id}$ and an arbitrary permutation $\sigma_2$~.

For any $\boldsymbol\theta_1 = (\boldsymbol\beta_{1}^T, \ldots, \boldsymbol\beta_{J-1}^T)^T$ with $\sigma_1 = {\rm id}$, we let  $\boldsymbol\theta_2 = (\boldsymbol\beta_{21}^T, \ldots, $ $\boldsymbol\beta_{2,J-1}^T)^T$ for $\sigma_2$, where
\begin{equation}\label{eq:adjacent_npo_beta2j}
\boldsymbol\beta_{2j} = \left\{
\begin{array}{cl}
\sum_{l=\sigma_2(j)}^{\sigma_2(j+1)-1} \boldsymbol\beta_l & \mbox{ if } \sigma_2(j) < \sigma_2(j+1)\\
-\sum_{l=\sigma_2(j+1)}^{\sigma_2(j)-1} \boldsymbol\beta_l & \mbox{ if } \sigma_2(j) > \sigma_2(j+1)
\end{array}\right.
\end{equation}
For $j=1, \ldots, J-1$, 
\[\eta_{ij}(\boldsymbol\theta_2) = {\mathbf h}_j^T({\mathbf x}_i) \boldsymbol\beta_{2j} = \sum_{l=\sigma_2(j)}^{\sigma_2(j+1)-1} {\mathbf h}_l^T({\mathbf x}_i) \boldsymbol\beta_l = \sum_{\sigma_2(j)}^{\sigma_2(j+1)-1} \eta_{il}(\boldsymbol\theta_1)
\] 
if $\sigma_2(j) < \sigma_2(j+1)$;  and $\eta_{ij}(\boldsymbol\theta_2) = -\sum_{l=\sigma_2(j+1)}^{\sigma_2(j)-1} \eta_{il}(\boldsymbol\theta_1)$ if $\sigma_2(j) > \sigma_2(j+1)$. It can be verified that \eqref{eq:adjacent_npo_beta2j} implies for any $1\leq j<k\leq J-1$,
\begin{equation}\label{eq:adjacent_npo_sum_beta2j}
\sum_{l=j}^k\boldsymbol\beta_{2j} = \left\{
\begin{array}{cl}
\sum_{l=\sigma_2(j)}^{\sigma_2(k+1)-1} \boldsymbol\beta_l & \mbox{ if } \sigma_2(j) < \sigma_2(k+1)\\
-\sum_{l=\sigma_2(k+1)}^{\sigma_2(j)-1} \boldsymbol\beta_l & \mbox{ if } \sigma_2(j) > \sigma_2(k+1)
\end{array}\right.
\end{equation}
Then for $j=1, \ldots, J-1$, 
\begin{eqnarray*}
\prod_{l=j}^{J-1}\frac{\rho_{il}(\boldsymbol\theta_2)}{1-\rho_{il}(\boldsymbol\theta_2)} &=& \exp\left\{\sum_{l=j}^{J-1} \eta_{il}(\boldsymbol\theta_2) \right\}\\  
&=& \left\{
\begin{array}{cl}
\prod_{l=\sigma_2(j)}^{\sigma_2(J)-1} \frac{\rho_{il}(\boldsymbol\theta_1)}{1-\rho_{il}(\boldsymbol\theta_1)} & \mbox{ if } \sigma_2(j) < \sigma_2(J)\\
\frac{1}{\prod_{l=\sigma_2(J)}^{\sigma_2(j)-1} \frac{\rho_{il}(\boldsymbol\theta_1)}{1-\rho_{il}(\boldsymbol\theta_1)}} & \mbox{ if } \sigma_2(j) > \sigma_2(J)
\end{array}\right.
\end{eqnarray*}
According to (2.2) and (2.3), it can be verified that $\pi_{ij}(\boldsymbol\theta_2) = \pi_{i\sigma_2(j)}(\boldsymbol\theta_1)$ for all $i=1, \ldots, m$ and $j=1, \ldots, J$. That is, (2.4) holds given $\boldsymbol\theta_1$~.

On the other hand, \eqref{eq:adjacent_npo_sum_beta2j} implies an inverse transformation from $\boldsymbol\theta_2$ to $\boldsymbol\theta_1$
\begin{equation}\label{eq:adjacent_inverse_transformation}
\boldsymbol\beta_j = \sum_{l=\sigma_2^{-1}(j)}^{\sigma_2^{-1}(j+1)} \boldsymbol\beta_{2l}
\end{equation}
with $j=1, \ldots, J-1$. That is, it is one-to-one from $\boldsymbol\theta_1$ to $\boldsymbol\theta_2$~. According to Theorem~2.1, ${\rm id} \sim \sigma_2$~.

Similar as in the proof of Theorem~2.5, we have $\sigma_1 \sim \sigma_2$ for any two permutations $\sigma_1$ and $\sigma_2$~. 
\end{proof}

\begin{proof} of Theorem~2.7: We frist verify condition~(2.4) for $\sigma_1={\rm id}$ and $\sigma_2=(J-1,J)$. In this case, given $\boldsymbol\theta_1 = (\boldsymbol\beta_1^T, \ldots, \boldsymbol\beta_{J-2}^T, \boldsymbol\beta_{J-1}^T)^T$ for $\sigma_1$, we let $\boldsymbol\theta_2 = (\boldsymbol\beta_1^T, \ldots, \boldsymbol\beta_{J-2}^T, -\boldsymbol\beta_{J-1}^T)^T$ for $\sigma_2$. Then $\eta_{ij}(\boldsymbol\theta_2) = {\mathbf h}_j^T({\mathbf x}_i) \boldsymbol\beta_j = \eta_{ij}(\boldsymbol\theta_1)$ for $j=1, \ldots, J-2$, and $\eta_{i,J-1}(\boldsymbol\theta_2) = -{\mathbf h}_{J-1}^T({\mathbf x}_i) \boldsymbol\beta_{J-1} = -\eta_{i,J-1}(\boldsymbol\theta_1)$. We further obtain $\rho_{ij}(\boldsymbol\theta_2) = \rho_{ij}(\boldsymbol\theta_1)$ for $j=1, \ldots, J-2$ and $\rho_{i,J-1}(\boldsymbol\theta_2) = 1- \rho_{i,J-1}(\boldsymbol\theta_1)$. According to (2.2) and (2.3), we obtain $\pi_{ij}(\boldsymbol\theta_2) = \pi_{ij}(\boldsymbol\theta_1)$ for $j=1, \ldots, J-2$; $\pi_{i,J-1}(\boldsymbol\theta_2) = \pi_{iJ}(\boldsymbol\theta_1)$; and $\pi_{iJ}(\boldsymbol\theta_2) = \pi_{i,J-1}(\boldsymbol\theta_1)$. That is, (2.4) holds given $\boldsymbol\theta_1$, which also holds given $\boldsymbol\theta_2$ since it is one-to-one from $\boldsymbol\theta_1$ to $\boldsymbol\theta_2$~. According to Theorem~2.1, ${\rm id} \sim (J-1,J)$ and thus $\sigma_1 \sim \sigma_1 (J-1,J)$ for any permutation $\sigma_1$~. 
\end{proof}

\begin{proof} of Lemma~2:
It is well known that for each $i=1, \ldots, m$, $\left(\frac{Y_{i1}}{N_i}, \ldots, \frac{Y_{iJ}}{N_i}\right)$ maximizes $\sum_{j=1}^J Y_{ij} \log \pi_{ij}$ as a function of $(\pi_{i1}, \ldots, \pi_{iJ})$ under the constraints $\sum_{j=1}^J \pi_{ij} = 1$ and $\pi_{ij}\geq 0$, $j=1, \ldots, J$ (see, for example, Section~35.6 of \cite{johnson1997discrete}). If $\hat{\boldsymbol\theta} \in \boldsymbol\Theta$ and $\hat\sigma\in {\mathcal P}$ satisfy $\pi_{i\hat\sigma^{-1}(j)}(\hat{\boldsymbol\theta}) = \frac{Y_{ij}}{N_i}$ for all $i$ and $j$, then $\{\pi_{i\hat\sigma^{-1}(j)}(\hat{\boldsymbol\theta})\}_{ij}$ maximizes $\sum_{i=1}^m \sum_{j=1}^J Y_{ij} \log\pi_{i\sigma^{-1}(j)}(\boldsymbol\theta)$, which implies $(\hat{\boldsymbol\theta}, \hat\sigma)$ maximizes $l_N(\boldsymbol\theta, \sigma)$ and thus $l(\boldsymbol\theta, \sigma)$.
\end{proof}

\begin{proof} of Lemma~3: 
According to the strong law of large numbers (see, for example, Chapter~4 in \cite{ferguson1996course}), $\frac{N_i}{N} = N^{-1} \sum_{l=1}^N {\mathbf 1}_{\{X_l={\mathbf x}_i\}} \rightarrow E({\mathbf 1}_{\{X_l={\mathbf x}_i\}}) = \frac{n_i}{n}$ almost surely, as $N\rightarrow \infty$, for each $i=1, \ldots, m$. Since $n_0\geq 1$, it can be verified that $\min\{N_1, \ldots, N_m\} \rightarrow \infty$ almost surely, as $N\rightarrow\infty$. Similarly, we have $\frac{Y_{ij}}{N_i} \rightarrow \pi_{i\sigma_0^{-1}(j)}(\boldsymbol\theta_0)$ almost surely, as $N_i\rightarrow \infty$, for each $i=1, \ldots, m$ and $j=1, \ldots, J$. Then as $N$ goes to infinity,
\begin{eqnarray*}
N^{-1}l_N(\boldsymbol\theta_0, \sigma_0) &=& \frac{1}{N}\sum_{i=1}^m \sum_{j=1}^J Y_{ij} \log\pi_{i\sigma_0^{-1}(j)}(\boldsymbol\theta_0)\\
&=& \sum_{i=1}^m \sum_{j=1}^J \frac{N_i}{N} \cdot \frac{Y_{ij}}{N_i} \log\pi_{i\sigma_0^{-1}(j)}(\boldsymbol\theta_0)\\
&\stackrel{a.s.}{\longrightarrow} & \sum_{i=1}^m \sum_{j=1}^J \frac{n_i}{n} \pi_{i\sigma_0^{-1}(j)}(\boldsymbol\theta_0) \log\pi_{i\sigma_0^{-1}(j)}(\boldsymbol\theta_0)\\
&=& \sum_{i=1}^m \sum_{j=1}^J \frac{n_i}{n} \pi_{ij}(\boldsymbol\theta_0) \log\pi_{ij}(\boldsymbol\theta_0)\> =\> l_0\> < \> 0
\end{eqnarray*}
\end{proof}

\begin{proof} of Theorem~3.1:
First we claim that for large enough $N$, all $i=1, \ldots, m$ and $j=1, \ldots, J$, $0 > \log\pi_{i\hat\sigma_N^{-1}(j)} (\hat{\boldsymbol\theta}_N) \geq \frac{2n(l_0-1)}{n_0\pi_0}$, which is a finite constant. Actually, since $(\hat{\boldsymbol\theta}_N, \hat{\sigma}_N)$ is an MLE, we have $N^{-1} l_N(\hat{\boldsymbol\theta}_N, \hat{\sigma}_N) \geq N^{-1}l_N(\boldsymbol\theta_0, \sigma_0)$ for each $N$. According to Lemma~3, $N^{-1}l_N(\boldsymbol\theta_0, \sigma_0) \rightarrow l_0$ almost surely, then $N^{-1} l_N(\hat{\boldsymbol\theta}_N, \hat{\sigma}_N) > l_0 -1$ for large enough $N$ almost surely. On the other hand, since $\log \pi_{i\hat\sigma_N^{-1}(j)}(\hat{\boldsymbol\theta}_N) < 0$ for all $i$ and $j$, then 
\[
N^{-1} l_N(\hat{\boldsymbol\theta}_N, \hat{\sigma}_N) = \sum_{i=1}^m \sum_{j=1}^J \frac{N_i}{N}\cdot \frac{Y_{ij}}{N_i} \log \pi_{i\hat\sigma_N^{-1}(j)}(\hat{\boldsymbol\theta}_N) < \frac{N_i}{N}\cdot \frac{Y_{ij}}{N_i} \log \pi_{i\hat\sigma_N^{-1}(j)}(\hat{\boldsymbol\theta}_N)
\]
for each $i$ and $j$. Since $\frac{N_i}{N} \rightarrow \frac{n_i}{n}$ almost surely and $\frac{Y_{ij}}{N_i} \rightarrow \pi_{i\sigma_0^{-1}(j)}(\boldsymbol\theta_0)$ almost surely, then $\frac{N_i}{N}\cdot \frac{Y_{ij}}{N_i} \log \pi_{i\hat\sigma_N^{-1}(j)}(\hat{\boldsymbol\theta}_N) < \frac{1}{2}\cdot \frac{n_i}{n} \pi_{i\sigma_0^{-1}(j)}(\boldsymbol\theta_0) \log \pi_{i\hat\sigma_N^{-1}(j)}(\hat{\boldsymbol\theta}_N)$ for large enough $N$ almost surely. Then we have 
\[
0 > \frac{1}{2}\cdot \frac{n_i}{n} \pi_{i\sigma_0^{-1}(j)}(\boldsymbol\theta_0) \log \pi_{i\hat\sigma_N^{-1}(j)}(\hat{\boldsymbol\theta}_N) > l_0-1
\]
almost surely for large enough $N$ and each $i$ and $j$. Since $l_0-1 < 0$, we further have almost surely for large enough $N$,
\[
0 > \log \pi_{i\hat\sigma_N^{-1}(j)}(\hat{\boldsymbol\theta}_N) > \frac{l_0-1}{\frac{1}{2}\cdot \frac{n_i}{n} \pi_{i\sigma_0^{-1}(j)}(\boldsymbol\theta_0)} \geq \frac{l_0-1}{\frac{1}{2}\cdot \frac{n_0}{n} \pi_0} = \frac{2n(l_0-1)}{n_0\pi_0}
\]
Now we are ready to check the asymptotic difference between $N^{-1} l_N(\hat{\boldsymbol\theta}_N, \hat{\sigma}_N)$ and $N^{-1}l_N(\boldsymbol\theta_0, \sigma_0)$. According to Lemma~2 and its proof, $(\boldsymbol\theta_0, \sigma_0)$ maximizes $\sum_{i=1}^m \sum_{j=1}^J \frac{N_i}{N}$ $\pi_{i\sigma_0^{-1}(j)}(\boldsymbol\theta_0) \log \pi_{i\sigma^{-1}(j)}(\boldsymbol\theta)$. Then
\begin{eqnarray*}
0 &\leq & N^{-1}[l_N(\hat{\boldsymbol\theta}_N, \hat\sigma_N) - l_N(\boldsymbol\theta_0, \sigma_0)]\\
&=& \sum_{i=1}^m \sum_{j=1}^J \frac{N_i}{N} \pi_{i\sigma_0^{-1}(j)}(\boldsymbol\theta_0) \log \pi_{i\hat\sigma_N^{-1}(j)}(\hat{\boldsymbol\theta}_N)\\
& - &\sum_{i=1}^m \sum_{j=1}^J \frac{N_i}{N} \pi_{i\sigma_0^{-1}(j)}(\boldsymbol\theta_0) \log \pi_{i\sigma_0^{-1}(j)}(\boldsymbol\theta_0)\\
& + & \sum_{i=1}^m \sum_{j=1}^J \frac{N_i}{N} \left[\frac{Y_{ij}}{N_i} - \pi_{i\sigma_0^{-1}(j)}(\boldsymbol\theta_0)\right] \cdot \left[\log \pi_{i\hat\sigma_N^{-1}(j)}(\hat{\boldsymbol\theta}_N) - \log \pi_{i\sigma_0^{-1}(j)}(\boldsymbol\theta_0)\right]\\
& \leq & \sum_{i=1}^m \sum_{j=1}^J \frac{N_i}{N} \left[\frac{Y_{ij}}{N_i} - \pi_{i\sigma_0^{-1}(j)}(\boldsymbol\theta_0)\right] \cdot \left[\log \pi_{i\hat\sigma_N^{-1}(j)}(\hat{\boldsymbol\theta}_N) - \log \pi_{i\sigma_0^{-1}(j)}(\boldsymbol\theta_0)\right]
\end{eqnarray*}
Then for large enough $N$, we have almost surely
\begin{eqnarray*}
& & \frac{1}{N} |l_N(\hat{\boldsymbol\theta}_N, \hat\sigma_N) - l_N(\boldsymbol\theta_0, \sigma_0)|\\
&\leq & \sum_{i=1}^m \sum_{j=1}^J \left|\frac{Y_{ij}}{N_i} - \pi_{i\sigma_0^{-1}(j)}(\boldsymbol\theta_0)\right| \cdot \left[\frac{-2n(l_0-1)}{n_0\pi_0} -\log \pi_0\right]
\end{eqnarray*}
Since $\frac{Y_{ij}}{N_i} \rightarrow \pi_{i\sigma_0^{-1}(j)}(\boldsymbol\theta_0)$ almost surely for each $i$ and $j$, then $N^{-1}|l_N(\hat{\boldsymbol\theta}_N, \hat\sigma_N)$ $-$ $l_N(\boldsymbol\theta_0, \sigma_0)| $ $\rightarrow 0$ almost surely as $N$ goes to infinity. The rest parts of the theorem are straightforward.
\end{proof}

\begin{proof} of Corollary~1:
Since ${\rm AIC} - {\rm AIC}(\boldsymbol\theta_0, \sigma_0) = {\rm BIC} - {\rm BIC}(\boldsymbol\theta_0, \sigma_0) = -2 l(\hat{\boldsymbol\theta}_N, \hat{\sigma}_N) + 2 l(\boldsymbol\theta_0, \sigma_0) = -2 l_N(\hat{\boldsymbol\theta}_N, \hat{\sigma}_N) + 2 l_N(\boldsymbol\theta_0, \sigma_0)$, the conclusion follows directly by Theorem~3.1.
\end{proof}

\end{document}